\documentclass[reqno]{amsart}
\usepackage{hyperref}
\usepackage{color}
\usepackage{mathrsfs}
\usepackage{graphicx}
\usepackage{amssymb, amsmath, amsthm}
\usepackage{amsfonts}
\usepackage{amssymb}
\usepackage{float}
\usepackage{mathrsfs}
\usepackage{enumitem} 
\hypersetup{
    bookmarks=true,         
    pdffitwindow=false,     
    pdfstartview={FitH},    
    colorlinks=false,       
}

\newtheorem{theorem}{Theorem}

\newtheorem{definition}[theorem]{Definition}

\newtheorem{lemma}[theorem]{Lemma}

\newtheorem{proposition}[theorem]{Proposition}
\newtheorem{remark}[theorem]{Remark}

\newcommand{\dis}{\displaystyle}

\newcommand{\divv}{\text{\rm div}}

\newcommand{\cA}{\tilde{\mathcal{A}}}

\newcommand{\cH}{\tilde{\mathcal{H}}}

\newcommand{\R}{\mathbb R}

\newcommand{\intox}{\int_{\R^3}}
\newcommand{\intoxtt}{\int_0^t\!\!\!\int_{\R^3}}
\newcommand{\intoxt}{\int_0^T\!\!\!\int_{\R^3}}
\newcommand{\intoxtl}{\int_1^t\!\!\!\int_{\R^3}}
\newcommand{\intoxts}{\int_0^{t_2}\!\!\!\intox}

\numberwithin{equation}{section}
\numberwithin{theorem}{section}

\numberwithin{figure}{section}

\begin{document}

%
%
%
%
%
%
%
%
%




\title[Uniqueness of weak solutions]
 {Existence and uniqueness of low-energy weak solutions to the compressible 3D magnetohydrodynamics equations}
 
\author{Anthony Suen} 

\address{Department of Mathematics and Information Technology\\The Education University of Hong Kong, Hong Kong}

\email{acksuen@eduhk.hk}

\dedicatory{To my family and my daughter Elisa}

\date{June 30, 2019}



\begin{abstract}
We prove the existence and uniqueness of weak solutions of the three dimensional compressible magnetohydrodynamics (MHD) equations. We first obtain the existence of weak solutions with small $L^2$-norm which  may display codimension-one discontinuities in density, pressure, magnetic field and velocity gradient. The weak solutions we consider here exhibit just enough regularity and structure which allow us to develop uniqueness and continuous dependence theory for the compressible MHD equations. Our results generalise and extend those for the intermediate weak solutions of compressible Navier-Stokes equations. 
\end{abstract}

\maketitle
\section{Introduction}

Magnetohydrodynamics (MHD) studies the dynamics of electrically conducting fluids under the influence of magnetic fields. There are many examples of conducting fluids, including plasmas, liquid metals, electrolytes, etc. The main idea of magnetohydrodynamics is that conducting fluids can support magnetic fields. More precisely, magnetic fields can induce currents in a moving conducting fluid, which in turn create forces on the fluid and also change the magnetic fields themselves. The subject of magnetohydrodynamics unites classical fluid dynamics with electrodynamics, and references can be found in \cite{Alfven42}, \cite{cabannes70}, \cite{Davidson01}, \cite{DucomentFeireisl06}.

In this present work, we focus on the following compressible barotropic model:
\begin{align}
 \rho_t + \divv (\rho u) &=0, \label{MHD1}\\
(\rho u^j)_t + \divv (\rho u^j u) + P(\rho)_{x_j} + ({\textstyle\frac{1}{2}}|B|^2)_{x_j}-\divv (B^j B)&= \mu \Delta u^j +
\lambda \, \divv \,u_{x_j},\label{MHD2}\\ 
B^{j}_{t} + \divv (B^j u - u^j B) &=\nu\Delta B^j,\label{MHD3}\\
\divv\,(B) &= 0,\label{MHD4}
\end{align}
 with given initial data
\begin{align}
(\rho,u,B)(x,0)&=(\rho_0,u_0,B_0)(x).\label{initial data}
\end{align}
For our barotropic model, the temperature is taken to be constant and the state of fluid motion is specified by three physical quantities: density $\rho$, velocity $u=(u^1,u^2,u^3)\in\R^3$ and magnetic field $B=(B^1,B^2,B^3)\in\R^3$. These quantities are all functions of the spatial coordinate $x\in\R^3$ and time $t\ge0$. $P=P(\rho)$ is the pressure which is an increasing function of $\rho$. $\mu$, $\lambda$ are positive viscosity coefficients and $\nu$ is the magnetic diffusivity. For a compressible barotropic flow, the equations of dynamics are given by the Navier-Stokes equations \eqref{MHD1}-\eqref{MHD2} which express the conservation of mass and conservation of momentum respectively. On the other hand, Maxwell's equations \eqref{MHD3}-\eqref{MHD4} govern the electromagnetic phenomena of the conducting fluid and the dynamics of magnetic fields. Therefore, by combining the compressible Navier-Stokes equations with Maxwell's equations, we obtain equations \eqref{MHD1}-\eqref{MHD4} which model the macroscopic behavior of electrically conducting fluids. We refer to Cabannes \cite{cabannes70}, Biskamp \cite{Biskamp97} and Freist\"{u}hler \cite{Freistuhler93} for more detailed discussions and derivation of the system \eqref{MHD1}-\eqref{MHD4}.

The global well-posedness of the system \eqref{MHD1}-\eqref{MHD4} is an active topic in mathematics, and the cases $\nu>0$ and $\nu=0$ are both of interest. When the magnetic diffusivity $\nu$ is taken to be positive, different types of solutions to \eqref{MHD1}-\eqref{MHD4} are proved to exist for all time:

\noindent{(a)} The first type of solutions to \eqref{MHD1}-\eqref{MHD4} are the {\it small-smooth solutions}. More precisely, Kawashima \cite{kawashima83} proved the global-in-time existence of $H^3$ solutions for \eqref{MHD1}-\eqref{MHD4} when the initial data was taken to be small in $H^3$ modulo a constant state. His analysis consists an iterative procedure based on asymptotic decay rates for the corresponding linearised equations. The major weakness of small-smooth solutions is that they do not exhibit nonlinear effects and tell us relatively little about the fluid flow. 

\noindent{(b)} The second type of solutions to \eqref{MHD1}-\eqref{MHD4} are the {\it large-energy weak solutions}. In this category, solutions are proved to exist for initial data with arbitrarily large energy and nonnegative density, which can be achieved by showing that sequences of approximate solutions with uniform energy estimates and entropy estimates have strongly converging subsequences. These results were obtained by Hu and Wang \cite{huwang08}-\cite{huwang10} and Sart \cite{sart09} which generalised the previous results proved by Lions \cite{lions98} and Feireisl \cite{feireisl02}-\cite{feireisl04} for compressible Navier-Stokes system. Large-energy weak solutions by their very nature possess very little regularity, which may even include some non-physical solutions (see \cite{hoff11} and \cite{hoffserre91} for related discussions). 

\noindent{(c)} Apart from those two types of solutions as mentioned in (a) and (b), Suen and Hoff \cite{suenhoff12} proved the global-in-time existence of {\it intermediate weak solutions} which was an extension of the intermediate regularity class of solutions for compressible Navier-Stokes system introduced by Hoff \cite{hoff95}-\cite{hoff06}. Such intermediate regularity class of solutions has rich physical and mathematical meanings compared to other solution classes. In this category, initial data is assumed to be small in some weak norms ($L^2$) with nonnegative and essentially bounded initial densities. From the results obtained by Hoff and Santos \cite{hoffsantos08} for the Navier-Stokes system, it can be seen that solutions may exhibit discontinuities in density and velocity gradient across hypersurfaces in $\R^2$ or $\R^3$, which is not observable from small-smooth solutions mentioned in (a). On the other hand, the solutions would still have enough regularity for the development of a uniqueness and continuous dependence theory \cite{hoff06} which seems unreachable within or from the very weak framework used by Lions and Feireisl mentioned in (b).

In light of (c) as described above, the main goal of the present work is therefore to address the global-in-time existence and uniqueness of intermediate weak solutions of the system \eqref{MHD1}-\eqref{MHD4}. The novelties of this current work are as follows:

\noindent{1.} We strengthen the results obtained in Suen-Hoff \cite{suenhoff12}, in the way that we show the details of the $s$-dependence of various smoothing rates near $t=0$ resulting from the hypothesis that $u_0,B_0\in H^s$ for $s\in(\frac{1}{2},1]$. Such regularity requirement on $u_0,B_0$ is crucial in obtaining uniqueness of the weak solutions of \eqref{MHD1}-\eqref{MHD4}. It also matches with the results given in Hoff \cite{hoff02} for Navier-Stokes equations.

\noindent{2.} We obtain new estimates on various auxiliary functionals which are important in controlling the strong coupling effects between density, velocity and magnetic fields. Those estimates will be used in proving both existence and uniqueness of weak solutions of \eqref{MHD1}-\eqref{MHD4}. 

\noindent{3.} We successfully extend the uniqueness and continuous dependence theory given in \cite{hoff06} for compressible Navier-Stokes system to compressible MHD system \eqref{MHD1}-\eqref{MHD4}.

\medskip

We give a brief exposition on the analysis applied in this work. First of all, we introduce an important canonical variable associated with the system \eqref{MHD1}-\eqref{MHD4}, which is known as the {\it effective viscous flux}. To see how it works, by the Helmholtz decomposition of the mechanical forces, we can rewrite the momentum equation \eqref{MHD2} as follows (summation over $k$ is understood):
\begin{equation}\label{derivation for F}
\rho\dot u^j +( {\textstyle\frac{1}{2}}|B|^2)_{x_j}-\divv(B^jB)=F_{x_j}+\mu\omega^{j,k}_{x_k},
\end{equation}
where $\dot u^j=u^j_t+u\cdot u^j$ is the material derivative on $u^j$ and the {\it effective viscous flux} $F$ is defined by
\begin{equation}\label{definition of F}
F=(\mu+\lambda){\rm div}(u) - P(\rho) + P(\tilde{\rho}).
\end{equation}
Differentiating \eqref{derivation for F}, we obtain the following Poisson equation
\begin{equation}\label{poisson in F}
\Delta F = \divv(g),
\end{equation}
where $g^j=\rho\dot{u}^j+(\frac{1}{2}|B|^2)_{x_j}-\divv(B^jB)$. This Poisson equation \eqref{poisson in F} is thus the analog for compressible MHD of the well-known elliptic equation for pressure in incompressible flow. The effective viscous flux $F$ plays a crucial roll in the overall analysis:

\noindent{1.} The equation \eqref{derivation for F} expresses the acceleration density $\rho \dot u$ as the sum of the gradient of the scalar $F$ and the divergence-free vector field $\omega^{\cdot,k}_{x_k}$, modulo lower-order terms involving $B$. The skew-symmetry of $\omega$ insures that these two vector fields are orthogonal in $L^2(\R^3)$, so that $L^2$-bounds for the terms on the left side of \eqref{derivation for F} immediately give $L^2$ bounds for the gradients of both $F$ and $\omega$. These in turn will be used for controlling $\nabla u$ in $L^4$ when $u(\cdot,t)\notin H^2$.

\noindent{2.} With the help of the effective viscous flux $F$ on the mass equation \eqref{MHD1}, we can further rewrite the equation as follows:
\begin{equation*}
(\mu+\lambda)\frac{d}{dt}[\log\rho(x(t),t)-\log(\tilde\rho)]+P(\rho(x(t),t))-\tilde P=-F(x(t),t),
\end{equation*}
where $x(t)$ is an integral curve of $u$ and $\tilde\rho$ is some constant density. Upon integrating the above equation with respect to $t$ on some interval $[t_1,t_2]$, if $P$ is increasing, then the integral of $P-\tilde P$ on the left side gives a dissipative term. Hence it suffices to control the term $\int_{ t_0}^{t_1}F(x(\tau),\tau)d\tau$. If $\Gamma$ is the fundamental solution for the Laplace operator on $\R^3$, then from \eqref{poisson in F} we have
\begin{equation*}
F=\Gamma_{x_j}*\Big(\rho\dot{u}^j+(\frac{1}{2}|B|^2)_{x_j}-\divv(B^jB)\Big).
\end{equation*}
There is a cancellation between the material derivative on $u$ and the time integral in $\int_{ t_0}^{t_1}F(s)ds$, as a result we can obtain integrals in lower regularity and hence greater integrability in time. Such observation is essential in proving the pointwise bounds on the density, which allows us to obtain sufficient {\it a priori} bounds on the solutions.

\noindent{3.} One of the key step in proving uniqueness of weak solutions is to obtain a bound on $\int_0^t\|\nabla u(\cdot,\tau)\|_{L^\infty}d\tau$. Our attempt is to decompose $u$ as $u=u_{F}+u_{P}$, where $u_{F}$, $u_{P}$ satisfy
\begin{align*}
\left\{
 \begin{array}{lr}
(\mu+\lambda)\Delta u_{F}^{j}=F_{x_j} +(\mu+\lambda)\omega^{j,k}_{x_k}\\
(\mu+\lambda)\Delta u_P^{j}=(P-P(\tilde{\rho}))_{x_j}.\\
\end{array}
\right.
\end{align*}
Using the {\it a priori} bounds on the effective viscous flux $F$, we can bound the integral $\int_0^t\|\nabla u_F(\cdot,\tau)\|_{L^\infty}d\tau$ in terms of $F$. On the other hand, to bound the integral $\int_{0}^{t}||\nabla u_{P}(\cdot,\tau)||_{\infty}d\tau$, we point out that $(\rho-\tilde\rho)\in L^2\cap L^\infty$ is {\it not} sufficient for bounding $||\nabla u_{P}(\cdot,\tau)||_{\infty}$. However, if $P(\rho(\cdot,t))\in L^\infty$, then $u^j_P(\cdot,t)=(\mu+\lambda)^{-1}\Gamma_{x_j}*(P(\rho(\cdot,t))-\tilde P)$ is log-Lipschitz. This is sufficient to guarantee that the integral curve $x(\cdot,t)$ of $u=u_F+u_P$ (assuming that $u_F$ has enough regularity as claimed) is H\"{o}lder continuous. If we assume that the initial density is {\it piecewise} H\"{o}lder continuous, then using the mass equation \eqref{MHD1}, it implies that the density is also {\it piecewise} H\"{o}lder continuous for positive time. Hence with such improved regularity on the density, it allows us to obtain the desired bound on $\int_{0}^{t}||\nabla u_{P}(\cdot,\tau)||_{\infty}d\tau$.

\medskip

We now give a detailed formulation of our results. To begin with, we require that the viscosity constants $\mu$, $\lambda$, $\nu$ and pressure $P(\rho)$ satisfy 
\begin{equation}\label{condition on vis}
\mu,\lambda,\nu>0,\qquad\frac{\mu}{\lambda}>4,
\end{equation}
and
\begin{equation}\label{condition on pressure}
P'(\rho)>0,\;\rho>0.
\end{equation}

The weak solutions to \eqref{MHD1}-\eqref{MHD4} are defined as follows. 
\begin{definition}\label{definition of weak solution}
We let $\tilde\rho$ be a fixed, positive, constant reference density and we take $\tilde P=P(\tilde\rho)$. The weak solutions we study in this paper are defined as follows. A weak solution of the system \eqref{MHD1}-\eqref{MHD4} is a triple $(\rho,u,B)$ which satisfies 
\begin{itemize}
\item $(\rho-\tilde{\rho},\,\rho u, B)\in C([0,\infty);H^{-1}(\R^3))$ with $(\rho,u,B)|_{t=0}=(\rho_0,u_0,B_0)$;
\item $\nabla u,\nabla B\in L^2(\R^3\times(0,\infty))$;
\item $\divv (B)(\cdot,t)=0$ in ${\mathcal D}'(\R^3)$ for $t>0$;
\end{itemize}
and the following identities hold for times $t_2\ge t_1 \ge 0$ and $C^1$ test functions $\varphi$ having uniformly bounded support in $x$ for $t\in[t_1,t_2]$:
\begin{align}\label{weak sol 1}
\left.\int_{\R^3}\rho(x,\cdot)\varphi(x,\cdot)dx\right|_{t_1}^{t_2}=\int_{t_1}^{t_2}\!\!\!\int_{\R^3}(\rho\varphi_t + \rho u\cdot\nabla\varphi)dxd\tau,
\end{align}
\begin{align}\label{weak sol 2}
\left.\int_{\R^3}(\rho u^{j})(x,\cdot)\varphi(x,\cdot)dx\right|_{t_1}^{t_2}=\int_{t_1}^{t_2}&\int_{\R^3}[\rho u^{j}\varphi_t + \rho u^{j}u\cdot\nabla\varphi + P(\rho)\varphi_{x_j}]dxd\tau\notag\\
&+ \int_{t_1}^{t_2}\!\!\!\int_{\R^3}\left[{\textstyle\frac{1}{2}}|B|^2\varphi_{x_j} - B^{j}B\cdot\nabla\varphi\right]dxd\tau\\& - \int_{t_1}^{t_2}\!\!\!\int_{\R^3}[(\mu+\lambda)\nabla u^{j}\cdot\nabla\varphi + \lambda(\divv(u))\varphi_{x_j}]dxd\tau,\notag
\end{align}
and
\begin{equation}\label{weak sol 3}
\left.\int_{\R^3}B^{j}(x,\cdot)\varphi(x,\cdot)dx\right|_{t_1}^{t_2}=\int_{t_1}^{t_2}\!\!\!\int_{\R^3}[(B^{j}u - u^{j}B)\cdot\nabla\varphi - \nu\nabla B^{j}\cdot\nabla\varphi]dxd\tau.
\end{equation}
\end{definition}
We adopt the usual notation for H\"older seminorms, namely for $v:\R^3\to \R^3$ and $\alpha \in (0,1]$, 
\hfill
$$\langle v\rangle^\alpha = \sup_{{x_1,x_2\in 
\R^3}\atop{x_1\not=x_2}}
{{|v(x_2) -v(x_1)|}\over{|x_2-x_1|^\alpha}}\,;$$
and for $v:Q\subseteq\R^3 \times[0,\infty)\to \R^3$ and $\alpha_1,\alpha_2 \in (0,1]$,
\hfill
$$\langle v\rangle^{\alpha_1,\alpha_2}_{Q} = \sup_{{(x_1,t_1),(x_2,t_2)\in 
Q}\atop{(x_1,t_1)\not=(x_2,t_2)}}
{{|v(x_2,t_2) - v(x_1,t_1)|}\over{|x_2-x_1|^{\alpha_1} + |t_2-t_1|^{\alpha_2}}}\,.$$
We give the definition of piecewise H\"{o}lder continuous as follows. We also refer to Hoff \cite{hoff02} for more details.
\begin{definition}\label{definition of piecewise C beta}
We say that a function $\phi(\cdot,t)$ is piecewise $C^{\beta(t)}$ if it has simple discontinuities across a $C^{\beta(t)+1}$ curve $\mathcal{C}(t):\mathcal{C}(t)=\{y(s,t):s\in I\subset\R\}$, where $\beta(t)>0$ is a function in $t$, $I$ is an open interval and the curve $\mathcal{C}(t)$ is the $u$-transport of $\mathcal C(0)$ given by:
$$y(s,t)=y(s,0)+\int_0^t u(y(s,\tau),\tau)d\tau.$$
Here $\mathcal{C}(0)$ is a $C^{\beta_0+1}$ curve with $\beta(0)=\beta_0>0$, which means that
\begin{equation*}
\mathcal{C}(0)=\{y_0(s):s\in\R\},
\end{equation*}
where $y(s,0)=y_0(s)$ is parameterised in arc length $s$ and $y_0$ is $C^{\beta_0+1}$. The complement
of $\mathcal{C}(0)$ consists of two disjoint, connected, open sets $\Omega_{+}(0)$ and $\Omega_{-}(0)$
with $\mathcal{C}(0)=\partial\Omega_{\pm}(0)$. 

We denote the norm $\|\phi(\cdot,t)\|_{C^{\beta(t)}_{pw}}$ by
\begin{equation*}
\|\phi(\cdot,t)\|_{C^{\beta(t)}_{pw}}=\|\phi(\cdot,t)\|_{L^\infty}+\sup_{x_1\neq x_2}\frac{|\phi(x_2,t)-\phi(x_1,t)|}{|x_2-x_1|^{\beta(t)}},
\end{equation*}
where the supremum is taken over points $x_1$, $x_2$ on the same side of $\mathcal{C}(t)$.
\end{definition}

We also make use of the following standard facts (see Ziemer \cite[Theorem~2.1.4, Remark~2.4.3, and Theorem~2.4.4]{ziemer89}, for example):
\begin{itemize}
\item First, given $r\in[2,6]$ there is a constant $C(r)$ such that for $w\in H^1 (\R^3)$,
\begin{equation}\label{Lr bound general}
\|w\|_{L^r(\R^3)} \le C(r) \left(\|w\|_{L^2(\R^3)}^{(6-r)/2r}\|\nabla w\|_{L^2(\R^3)}^{(3r-6)/2r}\right)
\end{equation}
and
\begin{equation}\label{holder bound general}
\langle w\rangle^\alpha\le C(r)\|\nabla w\|_{L^r(\R^3)},
\end{equation}
where $\alpha=1-3/r$;
\item for any $r\in (3,\infty)$ there is a constant $C(r)$ such that for $w\in W^{1,r}(\R^3)$,
\begin{equation}\label{L infty bound general}
\|w\|_{L^\infty (\R^3)} \le C(r) \|w\|_{W^{1,r}(\R^3)}.
\end{equation}
\end{itemize}

\medskip

We now state our main results. Theorem~\ref{Existence theorem} gives the existence of weak solutions to \eqref{MHD1}-\eqref{MHD4} with the $s$-dependence of smooth rates near $t=0$ (see \eqref{bound on weak solution} below). When the initial density $\rho_0$ is piecewise $C^{\beta_0}$ for some $\beta_0>0$ as defined in Definition~\ref{definition of piecewise C beta}, we prove that $\rho(\cdot,t)$ is piecewise $C^{\beta(t)}$ for $\beta(t)\in(0,\beta_0]$, which provides sufficient regularity in obtaining the bound on the time integral of $\|\nabla u(\cdot,t)\|_{L^\infty}$.

\begin{theorem}\label{Existence theorem} 
Fix constants $L,\rho_1,\rho_2,\tilde\rho>0$, $q>6$ and $s\in(\frac{1}{2},1]$ and assume that $\mu$, $\lambda$, $\nu$, $P$ satisfy \eqref{condition on vis}-\eqref{condition on pressure}. There exists positive constants $d$, $\theta$, $C$ such that if the initial data $(\rho_0,u_0,B_0)$ is given satisfying 
\begin{equation}\label{bound on initial density}
\rho_1\le\rho_0(x)\le \rho_2,\;x\in\R^3,
\end{equation}
\begin{equation}\label{boundedness on Lq initial data}
\|u_0\|_{L^q}+\|B_0\|_{L^q}\le L
\end{equation}
\begin{align}\label{definition of C0}
C_0&=\|u_0\|^2_{H^s}+\|B_0\|^2_{H^s}+\intox(|\rho-\tilde\rho|^2+|u_0|^2+|B_0|^2)dx\le d,
\end{align}
then the system \eqref{MHD1}-\eqref{MHD4} has a global weak solution $(\rho,u,B)$  in the sense of \eqref{weak sol 1}-\eqref{weak sol 3} on all of $\R^3\times[0,\infty)$. The solution satisfies the following:
\begin{equation}\label{rho in H-1}
\rho-\tilde\rho\in C([0,\infty);H^{-1}(\R^3));
\end{equation}
\begin{equation}\label{u B in L2}
u,B\in C([0,\infty);L^2(\R^3));
\end{equation}
\begin{equation}\label{nabla u B in L2}
\nabla u,\nabla B\in L^2(\R^3\times(0,\infty));
\end{equation}
\begin{equation}\label{u B in H1}
u(\cdot,t),B(\cdot,t)\in H^1 (\R^3),\;t>0;
\end{equation}
\begin{equation}\label{holder bound on u B}
\langle u\rangle^{\frac{1}{2},\frac{1}{4}}_{\R^3 \times [\tau,\infty)},\langle B\rangle^{\frac{1}{2},\frac{1}{4}}_{\R^3 \times [\tau,\infty)} \leq C(t)C_{0}^{\theta},\;t>0,
\end{equation}
where $C(\tau)$ may depend additionally on a positive lower bound for $\tau$, and the following bounds hold:
\begin{equation}\label{pointwise bound on rho theorem}
\frac{1}{2}\rho_1\le\rho(x,t)\le 2\rho_2;
\end{equation}
\begin{align}\label{bound on weak solution}
&\sup_{t>0}\intox\Big(|\rho-\tilde\rho|^2+|u|^2+|B|^2+\sigma^{1-s}(|\nabla u|^2+|\nabla B|^2)+\sigma^{2-s}(|\dot{u}|^2+|B_t|^2)\Big)dx\notag\\
&+\int_0^\infty\!\!\!\intox\Big(|\nabla u|^2+|\nabla B|^2+\sigma^{1-s}(|\dot{u}|^2+|B_t|^2)+\sigma^{2-s}(|\nabla\dot{u}|^2+|\nabla B_t|^2)\Big)dxd\tau\notag\\
&\le CC_0^{\theta},
\end{align}
where $\dot{u}=u_t + \nabla u\cdot u$ is the material derivative of $u$ and $\sigma=\min\{1,t\}$. 

Furthermore, for $\beta_0>0$, given a $C^{\beta_0+1}$ curve $y_0$ as described in Definition~\ref{definition of piecewise C beta}, if there exists $N>0$ such that 
\begin{equation}\label{piecewise holder for initial rho}
\|\rho_0(\cdot)-\tilde\rho\|_{C^{\beta_0}_{pw}}\le N,
\end{equation}
then for each $T>0$ and $t\in[0,T]$, there are $\beta(t)\in(0,\beta_0]$ and $C(N,T,C_0)>0$ such that we have
\begin{equation}\label{bound on time integral on nabla u}
\sup_{0\le \tau\le T}\|\rho(\cdot,t)-\tilde\rho\|_{C^{\beta(t)}_{pw}}+\int_0^T\|\nabla u(\cdot,\tau)\|_{L^\infty}d\tau\le C(N,T,C_0).
\end{equation}
\end{theorem}
\begin{remark}
We point out that the piecewise $C^{\beta_0}$-norm of $\rho_0$ is required to be bounded but not necessary small. This is different from the case of Hoff \cite{hoff02}, in which the author imposed a smallness assumption on $\|\rho_0\|_{C^{\beta_0}_{pw}}$ and proved that $\|\rho(\cdot,t)\|_{C^{\beta_0}_{pw}}$ remains small for all $t>0$. The key observation in our present work is that, without the smallness assumption on $\|\rho_0\|_{C^{\beta_0}_{pw}}$, we are able to show that there exists $\beta(t)\in(0,\beta_0]$ such that $\|\rho(\cdot,t)\|_{C^{\beta(t)}_{pw}}$ remains finite for finite time.
\end{remark}

Once we obtain Theorem~\ref{Existence theorem}, we address the uniqueness of weak solutions given in Theorem~\ref{Existence theorem} which can be summarised as follows. Theorem~\ref{Uniqueness theorem} illustrates the continuous dependence on the initial data of weak solutions, which generalises the results in Hoff \cite{hoff06} for compressible Navier-Stokes equations. 
\begin{theorem}\label{Uniqueness theorem} 
Fix constants $N,L,\rho_1,\rho_2,\tilde\rho>0$, $q>6$ and $s\in(\frac{1}{2},1]$ and assume that $\mu$, $\lambda$ satisfy \eqref{condition on vis} and $P$ satisfies
\begin{equation}\label{isothermal}
P(\rho)=K\rho
\end{equation}
for some constant $K$. Assume that $(\rho_0,u_0,B_0)$ and $(\bar{\rho}_0,\bar{u}_0,\bar{B}_0)$ are functions satisfying \eqref{bound on initial density}-\eqref{definition of C0} and \eqref{piecewise holder for initial rho} as in Theorem~\ref{Existence theorem}. Then for each $T>0$, there exists $C(T)>0$ such that if $(\rho,u,B)$ and $(\bar{\rho},\bar{u},\bar{B})$ are weak solutions to \eqref{MHD1}-\eqref{MHD4} as described in Theorem~\ref{Existence theorem} with initial data $(\rho_0,u_0,B_0)$ and $(\bar{\rho}_0,\bar{u}_0,\bar{B}_0)$ respectively, then we have
\begin{align}\label{estimate on difference of solutions}
&\Big(\intoxt(|u-\bar{u}|^2+|B-\bar{B}|^2)dxd\tau\Big)^\frac{1}{2}+\sup_{0\le \tau\le T}\|(\rho-\bar{\rho})(\cdot,t)\|_{H^{-1}}\notag\\
&\qquad\qquad\qquad\le C(T)\Big[\|\rho_0-\bar{\rho}_0\|_{L^2}+\|\rho_0u_0-\bar{\rho}_0\bar{u}_0\|_{L^2}+\|B_0-\bar{B}_0\|_{L^2}\Big].
\end{align}
\end{theorem}
\begin{remark}
Similar to the case as in Hoff \cite{hoff06}, the condition \eqref{isothermal} on the pressure can be replaced by a more general one, namely
\begin{equation*}
\sup_{0\le \tau\le T}\Big\|\nabla\Big(\frac{P(\rho(\cdot,t))-P(\bar{\rho}(\cdot,t))}{\rho(\cdot,t)-\bar{\rho}(\cdot,t)}\Big)\Big\|_{L^3}<\infty.
\end{equation*}
\end{remark}

\medskip

The rest of the paper is organised as follows. In Section~\ref{a priori estimates}, we obtain {\it a priori} estimates for smooth solutions to \eqref{MHD1}-\eqref{MHD4}. In Section~\ref{proof of existence section}, we apply the estimates obtained in Section~\ref{a priori estimates} to prove Theorem~\ref{Existence theorem} and give the details in obtaining bound on the time integral of $\|\nabla u(\cdot,t)\|_{L^\infty}$. Finally in Section~\ref{proof of uniqueness section}, we address the uniqueness of weak solutions given in Theorem~\ref{Existence theorem} by making use of the Lagrangian coordinates (the integral curve of $u$) and bounds on some auxiliary functionals. 

\section{A priori estimates}\label{a priori estimates}

In this section, we obtain some {\it a priori} estimates for smooth local-in-time solutions $(\rho - \tilde{\rho},u,B)$ of \eqref{MHD1}-\eqref{MHD4}. We first recall the following local-in-time existence theorem which was proved by Kawashima \cite{kawashima83}:

\begin{theorem}\label{local existence theorem}
For a given initial data $(\rho_0 - \tilde\rho,u_0,B_0)\in H^3(\R^3)$, there exists $T>0$ and a solution $(\rho ,u,B)$ to \eqref{MHD1}-\eqref{MHD4} defined on $\R^3\times[0,T]$ such that 
\begin{equation}\label{smooth local solution 1}
\rho - \tilde{\rho}\in C([0,T];H^{3}(\R^3))\cap C^1 ([0,T];H^{2}(\R^3))
\end{equation}
and
\begin{equation}\label{smooth local solution 2}
u,B\in C([0,T];H^{3}(\R^3))\cap C^1 ([0,T];H^1 (\R^3))\cap L^2([0,T];H^4 (\R^3)).
\end{equation}
\end{theorem}

The estimates for $(\rho - \tilde{\rho},u,B)$ given in this section will be crucial in proving Theorem~\ref{Existence theorem}. The main goal is to prove the following theorem:

\begin{theorem}\label{a priori bounds finite time theorem}
Fix constants $L,\rho_1,\rho_2,\tilde\rho>0$, $q>6$ and $s\in(\frac{1}{2},1]$, and assume that $\mu$, $\lambda$, $P$ satisfy \eqref{condition on vis}-\eqref{condition on pressure}. Let the initial data $(\rho_0-\tilde\rho,u_0,B_0)\in H^3(\R^3)$ be given satisfying \eqref{bound on initial density}-\eqref{definition of C0}. There exists positive constants $d$, $\theta$, $C$ such that if $(\rho,u,B)$ is a solution of \eqref{MHD1}-\eqref{MHD4} on $\R^3\times [0,T]$ satisfying \eqref{smooth local solution 1}-\eqref{smooth local solution 2}, then we have 
\begin{equation}\label{bound on weak solution finite time}
\mathcal{A}(T)\le CC_0^{\theta}
\end{equation}
and 
\begin{equation}\label{pointwise bound on rho finite time}
\frac{1}{2}\rho_1\le\rho(x,t)\le 2\rho_2,\; (x,t)\in\R^3\times[0,T],
\end{equation}
where $\mathcal{A}(T)$ is given by
\begin{align}\label{def of A(t)}
\mathcal{A}(T)&=\sup_{0\le \tau\le T}\Big[\intox(|\rho-\tilde\rho|^2+|u|^2+|B|^2)\Big](x,\tau)dx\\
&+\sup_{0\le \tau\le T}\Big[\intox\sigma^{1-s}(|\nabla u|^2+|\nabla B|^2)+\sigma^{2-s}(|\dot{u}|^2+|B_t|^2)\Big](x,\tau)dx\notag\\
&+\intoxt\Big(|\nabla u|^2+|\nabla B|^2+\sigma^{1-s}(|\dot{u}|^2+|B_t|^2)+\sigma^{2-s}(|\nabla\dot{u}|^2+|\nabla B_t|^2)\Big)dxd\tau,\notag
\end{align}
and $\sigma(t)=\min\{1,t\}$.
\end{theorem}

The proof of Theorem~\ref{a priori bounds finite time theorem} will be carried out in a sequence of lemmas. We first establish the bound \eqref{bound on weak solution finite time} {\it under the assumption} that \eqref{pointwise bound on rho finite time} holds for the density $\rho$, which will be given in subsection~\ref{small time estimates} and subsection~\ref{large time estimates}. Then in subsection~\ref{Pointwise bound on rho subsection}, we close the estimates of Theorem~\ref{a priori bounds finite time theorem} by deriving pointwise bounds \eqref{pointwise bound on rho finite time} for $\rho$ under the smallness assumption on $C_0$. This gives an uncontingent estimate for $(\rho,u,B)$ and thereby proving Theorem~\ref{a priori bounds finite time theorem}.

Throughout this section, $C$ will denote a generic positive constant which depends on the same quantities as the constant $C$ in the statement of Theorem~\ref{a priori bounds finite time theorem} but independent of time $t$ and the regularity of initial data. 

We first recall the following estimates on the effective viscous flux $F$ which is defined in \eqref{definition of F}.

\begin{lemma}\label{estimate on effective viscous flux}
Assume that $\rho$ satisfies \eqref{pointwise bound on rho finite time}. For each $p>1$, there is a constant $C>0$ such that for all $t>0$, we have
\begin{equation}\label{bound on F}
\|F(\cdot,t)\|_{L^p}\le C\Big[\|\nabla u(\cdot,t)\|_{L^p}+\|(\rho-\tilde\rho)(\cdot,t)\|_{L^p}\Big],
\end{equation}
and
\begin{equation}\label{bound on nabla F}
\|\nabla F(\cdot,t)\|_{L^p}\le C\Big[\|\dot{u}(\cdot,t)\|_{L^p}+\|B\nabla B(\cdot,t)\|_{L^p}\Big].
\end{equation}
\end{lemma}
\begin{proof}
The assertion \eqref{bound on F} follows immediately from the definition of $F$, and the proof of \eqref{bound on nabla F} relies on the Poisson equation \eqref{poisson in F} and the Marcinkiewicz multiplier theorem (refer to Stein \cite{stein70}, pg. 96). 
\end{proof}

Using the estimates \eqref{bound on F}-\eqref{bound on nabla F} on $F$, we have the following estimates on $\nabla u$ and $\nabla\omega$:

\begin{lemma}\label{estimate on nabla u lemma}
Assume that $\rho$ satisfies \eqref{pointwise bound on rho finite time}. For each $p>1$, there is a constant $C>0$ depends on $p$ such that for all $t>0$, we have
\begin{align}\label{bound on nabla u}
\|\nabla u(\cdot,t)\|_{L^p}\le C\Big[\|F(\cdot,t)\|_{L^p}+\|\omega(\cdot,t)\|_{L^p}+\|(P-\tilde P)(\cdot,t)\|_{L^p}\Big],
\end{align}
\begin{align}\label{bound on omega}
\|\nabla\omega(\cdot,t)\|_{L^p}\le C\Big[\|\dot{u}(\cdot,t)\|_{L^p}+\|B\nabla B(\cdot,t)\|_{L^p}\Big].
\end{align}
\end{lemma}

\begin{proof} By the definition \eqref{definition of F} of $F$,
\begin{equation*}
(\mu+\lambda)\Delta u^j=F_{x_j}+(\mu+\lambda)\omega^{j,k}_{x_k}+(P-\tilde{P})_{x_j}.
\end{equation*}
Hence by differentiating and taking the Fourier transform on the above equation, we can apply Marcinkiewicz multiplier theorem in a similar as we did in Lemma~\ref{estimate on effective viscous flux} and \eqref{bound on nabla u} follows.

For the case of $\nabla\omega$, by direct computation, we have
\begin{equation*}
\mu\Delta\omega=(\rho\dot{u}^j)_{x_k}-(\rho\dot{u}^k)_{x_j}-(\nabla B^{j}\cdot B)_{x_k}+(\nabla B^{k}\cdot B)_{x_j},
\end{equation*}
and using the same argument as for $\nabla u$, \eqref{bound on omega} immediately follows.
\end{proof}

We now start giving the estimates on $(\rho,u,B)$ which will be used in deriving \eqref{bound on weak solution finite time}. We begin with the following $L^2$ estimates on $(\rho,u,B)$ for all $T>0$:

\begin{lemma}\label{L2 estimate}
Assume that $\rho$ satisfies \eqref{pointwise bound on rho finite time}. For $T>0$, we have
\begin{align}\label{L2 bound}
\sup_{0\le \tau\le T}\intox(|\rho-\tilde\rho|^2+\rho|u|^2+|B|^2)dx+\intoxt(|\nabla u|^2+|\nabla B|^2)dxd\tau\le CC_0.
\end{align}
\end{lemma}

\begin{proof}
The bound \eqref{L2 bound} follows from the standard energy balance equation, namely
\begin{align*}
&\left.\int_{\R^3}\left({\textstyle\frac{1}{2}}\rho|u|^{2}+\mathcal{G}(\rho)\right)dx\right |_{0}^{t} + \left.\int_{\R^3}{\textstyle\frac{1}{2}}|B|^2 dx\right |_{0}^t \\
&\qquad+ \int_{0}^{t}\!\!\!\int_{\R^3}\left(\mu|\nabla u|^{2} + \lambda(\divv (u))^{2}+\nu|\nabla B|^2\right)dxd\tau=0,
\end{align*}
where $\dis \int_{\R^d}\mathcal{G}(\rho)dx=\int_{\R^d}\left(\rho\int_{\tilde{\rho}}^{\rho}\tau^{-2}( P(\tau)-P(\tilde\rho))d\tau\right)dx$ is comparable to the $L^2(\R^3)$ norm of $(\rho-\tilde\rho)$ (see \cite{hoff95} for related discussion).  
\end{proof}

To proceed further, we have to obtain higher order estimates on $u$ and $B$. Due to the intricate coupling effects between $u$ and $B$, we subdivide the estimates into two cases namely $T\le1$ and $T>1$. These will be illustrated in subsection~\ref{small time estimates} and subsection~\ref{large time estimates} as follows:

\subsection{Estimates on $u$ and $B$ for $T\le1$}\label{small time estimates} 

In this subsection, we obtain estimates on $u$ and $B$ for $T\le1$.
 We start with the following $L^6$ estimates on $u$ and $B$:
 
\begin{lemma}\label{L6 estimate}
Assume that $\rho$ satisfies \eqref{pointwise bound on rho finite time}. For $T\le1$, we have
\begin{align}\label{L6 bound}
\sup_{0\le \tau\le T}\intox(|u|^6+|B|^6)dx\le C_LC_0^{\theta_q},
\end{align}
where $C_L,\theta_q>0$ depends on $L$ and $q$ respectively.
\end{lemma}

\begin{proof}
We follow the computations given in \cite{suenhoff12} and obtain, for $0\le t\le T\le1$,
\begin{align*}
&\intox(|u|^6+|B|^6)dx\Big|_{\tau=0}^t+6\intoxt(\mu|u|^2|\nabla u|^2+\nu|B|^2|\nabla B|^2)dxd\tau\\
&\qquad+(-24\lambda+6\mu)\intoxt|u|^2|\nabla(|u|^2)|^2dxd\tau+6\nu\intoxt|B|^2|\nabla)|B|^2)|^2dxd\tau\\
&\le C\Big[\intoxt|\rho-\tilde\rho||\divv(|u|^4u)|dxd\tau+\intoxt|u|^4|u||\divv(BB^T)|dxd\tau\Big]\\
&\qquad+C\Big[\intoxt|u|^4|u||\nabla(\frac{1}{2}|B|^2)|dxd\tau+\intoxt|B|^4|B\cdot\divv(Bu^T-uB^T)|dxd\tau\Big].
\end{align*}
By the assumption \eqref{condition on vis}, the term involving $(-24\lambda+6\mu)$ is positive, while the term $C\intox(|u_0|^6+|B_0|^6)dx$ can be bounded in terms of $C_0$ and $L$ by interpolation and assumption \eqref{boundedness on Lq initial data}. The rest of the analysis follows by a Gronw\"{a}ll-type argument (also see \cite{suenhoff12} for details) and we omit the details here. 
\end{proof}

Next we derive bounds for $u$ and $B$ in $L^\infty([0,T];H^1(\R^3))$ when $T\le1$.

\begin{lemma}\label{H1 estimates}
Assume that $\rho$ satisfies \eqref{pointwise bound on rho finite time} and $C_0\ll1$. For $T\le 1$ and $s\in[0,1]$, we have
\begin{align}\label{H1 bound}
\sup_{0\le \tau\le T}\tau^{1-s}\intox(|\nabla u|^2+|\nabla B|^2)dx+\intox \tau^{1-s}(|\dot{u}|^2+|B_t|^2)dx\le CC_0.
\end{align}
\end{lemma}

\begin{proof}
We apply the interpolation argument as given in Hoff \cite{hoff02}. We define differential operators $\mathcal{L}_u$, $\mathcal{L}_B$ acting on functions $w:\R^3\times[0,\infty)\rightarrow\R^3$ by
\begin{align*}
(\mathcal{L}_u w)^j&=(\rho w^j)_t+\divv(\rho w^j u)+({\textstyle\frac{1}{2}}|B|^2)_{x_j}-\divv (B^j B)-(\mu \Delta u^j +
\lambda \, \divv \,u_{x_j}),\\
(\mathcal{L}_B w)^j&=w^{j}_{t} + \divv (w^j u - u^j w)-\nu\Delta w^j.
\end{align*}
Then we define $w_1$ and $w_2$ by
\begin{align*}
\mathcal{L}_u w_1=0,\,\,\,w_1(x,0)=w_{10}(x),\qquad\mathcal{L}_u w_2=-\nabla P(\rho),\,\,\,w_2(x,0)=0,
\end{align*}
for a given $w_{10}$. Notice that $\mathcal{L}_B B=0$, and if $w_{10}=u_0$, then $w_1+w_2=u$. Using the energy estimate as obtained in Lemma~\ref{L2 estimate}, we have
\begin{align}\label{L2 estimate on w1}
&\sup_{0\le \tau\le T}\intox(|w_1(x,t)|^2+|B|^2)dx+\intoxt(|\nabla w_1|^2+|\nabla B|^2)dxd\tau\notag\\
&\qquad\le C\intox(|w_{10}|^2+|B_0|^2)dx,
\end{align}
as well as
\begin{align}\label{L2 estimate on w2}
&\sup_{0\le \tau\le T}\intox(|w_2(x,\tau)|^2+|B|^2)dx+\intoxt(|\nabla w_2|^2+|\nabla B|^2)dxd\tau\notag\\
&\qquad\le C\intox|B_0|^2dx+CT\sup_{0\le \tau\le T}\|(P-\tilde P)(\cdot,\tau)\|_{L^2}^2,
\end{align}

Also, for $k\in\{0,1\}$ and $0\le t\le T$, we have
\begin{align}\label{H1 estimate on w1 and B}
&\tau^k\intox(|\nabla w_1(x,\tau)|^2+|\nabla B|^2)dx\Big|_{\tau=0}^{\tau=t}+\intoxt\tau^k(\rho|\dot{w_1}|^2+|B_t|^2)dxd\tau\notag\\
&\le C\Big(\intoxt k\tau^{k-1}(|\nabla w_1|^2+|\nabla B|^2)dxd\tau+\intoxt\tau^{\frac{3k}{2}}(|\nabla w_1|^3+|\nabla B|^3)dxd\tau\notag\\
&\qquad+C\intoxt \tau^k(|\nabla B|^2|B|^2+|\nabla B|^2|u|^2+|\nabla u|^2|B|^2)dxd\tau,
\end{align}
and
\begin{align}\label{H1 estimate on w2}
&\intox|\nabla w_2(x,\tau)|^2dx\Big|_{\tau=0}^{\tau=t}+\intoxt\rho|\dot{w_2}|^2dxd\tau\notag\\
&\le C\Big(\Big|\intox(P-\tilde P)\divv (w_2)(x,\tau)dx\Big|_{\tau=0}^{\tau=t}\Big|\Big)\notag\\
&\qquad+C\Big(\intoxt\tau^{\frac{3k}{2}}(|\nabla w_2|^3+|\nabla B|^3)dxd\tau+\intoxt \tau^k|\nabla B|^2|B|^2dxd\tau\Big).
\end{align}
Using \eqref{L2 estimate on w1}-\eqref{L2 estimate on w2}, the terms 
$$\intoxt k\tau^{k-1}(|\nabla w_1|^2+|\nabla B|^2)dxd\tau$$
and
$$\Big|\intox(P-\tilde P)\divv (w_2)(x,\tau)dx\Big|_{\tau=0}^{\tau=t}\Big|$$
can be bounded in terms of $C_0$, namely
\begin{align*}
\Big|\intoxt k\tau^{k-1}(|\nabla w_1|^2+|\nabla B|^2)dxd\tau\Big|\le C\intoxt (|\nabla w_1|^2+|\nabla B|^2)dxd\tau\le CC_0,
\end{align*}
and 
\begin{align*}
&\Big|\intox(P-\tilde P)\divv (w_2)(x,\tau)dx\Big|_{\tau=0}^{\tau=t}\Big|\\
&\le C\Big(\intox|\rho-\tilde\rho|^2(x,t)dx\Big)^\frac{1}{2}\Big(\intox|\nabla w_2|^2(x,t)dx\Big)^\frac{1}{2}\le CC_0^\frac{1}{2}\Big(\intox|\nabla w_2|^2(x,t)dx\Big)^\frac{1}{2} .
\end{align*}
We now aim at controlling the higher terms as appeared on the right sides of \eqref{H1 estimate on w1 and B} and \eqref{H1 estimate on w2}. In view of the bound \eqref{L6 bound}, it suffices to consider \[\intoxt\tau^{\frac{3k}{2}}(|\nabla w_i|^3+|\nabla u|^3+|\nabla B|^3)dxd\tau\] for $i=1,2$. For the term involving $\nabla B$, using \eqref{Lr bound general}, we can estimate it as follows.
\begin{align*}
&\intoxtt \tau^{\frac{3k}{2}}|\nabla B|^3dxd\tau\\
&\le C\intoxtt \tau^{\frac{3k}{2}}\Big(\intox|\Delta B|^2\Big)^\frac{3}{4}\Big(\intox|\nabla B|^2\Big)^\frac{3}{4}dxd\tau\\ 
&\le C\int_0^t\tau^{\frac{3k}{2}}\Big(\intox(|B_t|^2+|\nabla B|^2|u|^2+|\nabla u|^2|B|^2)\Big)^\frac{3}{4}\Big(\intox|\nabla B|^2\Big)^\frac{3}{4}\\
&\le C\Big(\sup_{0\le \tau\le T}\tau^{k}\intox|\nabla B|^2dx\Big)^\frac{1}{2}\Big(\intoxtt \tau^{k}|B_t|^2dx\Big)^\frac{3}{4}\Big(\intoxtt |\nabla B|^2dx\Big)^\frac{1}{4}d\tau\\
&\qquad+C\int_0^t\tau^{\frac{3k}{2}}\Big(\intox|\nabla B|^3dx\Big)^\frac{1}{2}\Big(\intox|u|^6dx\Big)^\frac{1}{4}\Big(\intox|\nabla B|^2dx\Big)^\frac{3}{4}d\tau\\
&\qquad+C\int_0^t\tau^{\frac{3k}{2}}\Big(\intox|\nabla u|^3dx\Big)^\frac{1}{2}\Big(\intox|B|^6dx\Big)^\frac{1}{4}\Big(\intox|\nabla B|^2dx\Big)^\frac{3}{4}d\tau.
\end{align*}
Therefore, by the bounds \eqref{L2 bound} and \eqref{L6 bound}, we obtain
\begin{align*}
&\intoxtt \tau^{\frac{3k}{2}}|\nabla B|^3dxd\tau\\
&\le CC_0^\frac{1}{4}\Big(\sup_{0\le \tau\le T}\tau^{k}\intox|\nabla B|^2dx\Big)^\frac{1}{2}\Big(\intoxtt \tau^{k}|B_t|^2dxd\tau\Big)^\frac{3}{4}\\
&\qquad+CC_0^\frac{1}{4}C_0^\frac{1}{2}\Big(\sup_{0\le \tau\le T}\tau^{k}\intox|\nabla B|^2dx\Big)^\frac{1}{4}\Big(\intoxtt \tau^{\frac{3k}{2}}|\nabla B|^3dxd\tau\Big)^\frac{1}{2}\\
&\qquad+CC_0^\frac{1}{4}C_0^\frac{1}{2}\Big(\sup_{0\le \tau\le T}\tau^{k}\intox|\nabla B|^2dx\Big)^\frac{1}{4}\Big(\intoxtt \tau^{\frac{3k}{2}}|\nabla u|^3dxd\tau\Big)^\frac{1}{2}.
\end{align*}
To estimate the term $\intoxt t^\frac{3k}{2}|\nabla u|^3$, we apply \eqref{Lr bound general} and the bounds \eqref{bound on nabla F}-\eqref{bound on nabla u} to obtain
\begin{align*}
&\intoxtt \tau^{\frac{3k}{2}}|\nabla u|^3dxd\tau\\
&\le C\intoxtt \tau^{\frac{3k}{2}}(|F|^3+|\omega|^3+|P-\tilde P|^3)dxd\tau\\
&\le C\int_0^t \tau^{\frac{3k}{2}}\Big(\intox|F|^2dx\Big)^\frac{3}{4}\Big(\intox|\nabla F|^2dx\Big)^\frac{3}{4}d\tau+C\intoxtt \tau^{\frac{3k}{2}}(|\omega|^3+|P-\tilde P|^3)dxd\tau\\
&\le CC_0^\frac{1}{4}\Big(\sup_{0\le \tau\le T}\tau^{k}\intox|\nabla u|^2dx\Big)^\frac{1}{2}\Big(\intoxtt \tau^{k}|\dot{u}|^2dxd\tau\Big)^\frac{3}{4}+C_0^\frac{3}{4}\Big(\intoxtt \tau^{k}|\dot{u}|^2dxd\tau\Big)^\frac{3}{4}\\
&\qquad+CC_0^\frac{1}{4}C_0^\frac{1}{2}\Big(\sup_{0\le \tau\le T}\tau^{k}\intox|\nabla u|^2dx\Big)^\frac{1}{4}\Big(\intoxtt \tau^{\frac{3k}{2}}|\nabla B|^3dxd\tau\Big)^\frac{1}{2}\\
&\qquad+CC_0\Big(\intoxtt \tau^{\frac{3k}{2}}|\nabla B|^3dxd\tau\Big)^\frac{1}{2}+C\intoxtt \tau^{\frac{3k}{2}}(|\omega|^3+|P-\tilde P|^3)dxd\tau.
\end{align*}
The estimates for $w_1$ and $w_2$ are just similar. In view of the above, under suitable smallness assumption on the initial data, the integrals $\intoxt\tau^{\frac{3k}{2}}(|\nabla u|^3+|\nabla B|^3)$ can be absorbed by the left sides of \eqref{H1 estimate on w1 and B} and \eqref{H1 estimate on w2}. We treat $w_1$ and $w_2$ in a similar way and obtain the following estimates on $w_1$, $w_2$ and $B$:
\begin{align}
\sup_{0\le \tau\le T}\intox(|\nabla w_1|^2&+|\nabla B|^2)dx+\intoxt (|\dot{w_1}|^2+|B_t|^2)dxd\tau\notag\\
&\le C(\|w_{10}\|^2_{H^1}+\|B_0\|^2_{H^1}),\label{H1 bound on w1 and B smooth}\\
\sup_{0\le \tau\le T}t\intox(|\nabla w_1|^2&+|\nabla B|^2)dx+\intoxt \tau(|\dot{w_1}|^2+|B_t|^2)dxd\tau\notag\\
&\le C(\|w_{10}\|^2_{L^2}+\|B_0\|^2_{L^2}),\label{H1 bound on w1 and B not smooth}\\
\sup_{0\le \tau\le T}\intox|\nabla w_2|^2dx&+\intoxt|\dot{w_2}|^2dxd\tau\le CC_0.\label{H1 bound on w2}
\end{align}
Since the operators $\mathcal{L}_u$ and $\mathcal{L}_B$ are both linear, we can apply Riesz-Thorin interpolation to deduce from \eqref{H1 bound on w1 and B smooth}-\eqref{H1 bound on w1 and B not smooth} that
\begin{align}\label{Hs bound on w1 and B}
&\sup_{0\le t\le 1}\tau^{1-s}\intox(|\nabla w_1|^2+|\nabla B|^2)dx+\intoxt \tau^{1-s}(|\dot{w_1}|^2+|B_t|^2)dxd\tau\notag\\
\qquad&\le C(\|w_{10}\|^2_{H^s}+\|B_0\|^2_{H^s}).
\end{align}
By taking $w_{10}=u_0$, we conclude from \eqref{H1 bound on w2} and \eqref{Hs bound on w1 and B} that
\begin{align*}
&\sup_{0\le \tau\le T}\tau^{1-s}\intox(|\nabla u|^2+|\nabla B|^2)dx\\
&\qquad+\intoxt \tau^{1-s}(|\dot{u}|^2+|B_t|^2)dxd\tau\le CC_0.
\end{align*}
\end{proof}

Next we give the following auxiliary bounds which are useful in estimating some mixed terms in $u$ and $B$ when $T\le1$.

\begin{lemma}\label{auxiliary bound 1 lemma}
Assume that $\rho$ satisfies \eqref{pointwise bound on rho finite time} and $C_0\ll1$. For $T\le1$ and $s\in[0,1]$, we have
\begin{align}\label{auxiliary bound 1}
\intoxt \tau^{1-s}(|\nabla B|^2|B|^2+|\nabla u|^2|B|^2+|\nabla B|^2|u|^2)dxd\tau\le CC_0^{\theta},
\end{align}
for some $\theta>0$.
\end{lemma}

\begin{proof}
Using equation \eqref{MHD3} and the bounds \eqref{L2 bound}-\eqref{L6 bound} and \eqref{H1 bound}, we have
\begin{align*}
&\int_0^T\tau^{1-s}\intox|\nabla B|^2|B|^2dxd\tau\\
&\le\int_0^T\tau^{1-s}\Big(\intox|\nabla B|^3dx\Big)^\frac{2}{3}\Big(\intox|B|^6dx\Big)^\frac{1}{3}d\tau\\
&\le C\int_0^T\tau^{1-s}\Big(\intox|\nabla B|^2dx\Big)^\frac{1}{2}\Big(\intox|\Delta B|^2dx\Big)^\frac{1}{2}\Big(\intox|B|^6dx\Big)^\frac{1}{3}d\tau\\
&\le C\int_0^T\tau^{1-s}\Big(\intox|\nabla B|^2dx\Big)^\frac{1}{2}\Big(\intox(|B_t|^2+|\nabla u|^2|B|^2+|\nabla B|^2|u|^2)dx\Big)^\frac{1}{2}\Big(\intox|B|^6dx\Big)^\frac{1}{3}d\tau\\
&\le CC_0^\frac{1}{3}C_0^\frac{1}{2}\Big(C_0+\intoxt\tau^{1-s}(|\nabla u|^2|B|^2+|\nabla B|^2|u|^2)dxd\tau\Big)^\frac{1}{2}.
\end{align*}
Similarly, we have
\begin{align*}
&\int_0^T\tau^{1-s}\intox|\nabla B|^2|u|^2dxd\tau\\
&\le CC_0^\frac{1}{3}C_0^\frac{1}{2}\Big(C_0+\intoxt\tau^{1-s}(|\nabla u|^2|B|^2+|\nabla B|^2|u|^2)dxd\tau\Big)^\frac{1}{2}.
\end{align*}
Finally, for the term $\int_0^T\intox \tau^{1-s}|\nabla u|^2|B|^2$, using the bound \eqref{bound on nabla u} on $\nabla u$, 
\begin{align}\label{bound on nabla u B}
&\int_0^T\intox \tau^{1-s}|\nabla u|^2|B|^2dxd\tau\notag\\
&\le C\int_0^T\tau^{1-s}\Big(\intox|\nabla u|^3dx\Big)^\frac{2}{3}\Big(\intox|B|^6dx\Big)^\frac{1}{3}d\tau\notag\\
&\le CC_0^\frac{1}{3}\int_0^T\tau^{1-s}\Big(\intox(|F|^3+|\omega|^3+|\rho-\tilde\rho|^3)dx\Big)^\frac{2}{3}d\tau.
\end{align}
To bound the term involving $F$ in \eqref{bound on nabla u B}, using the estimates \eqref{bound on F}-\eqref{bound on nabla F}, we have
\begin{align*}
&\int_0^T\Big(\intox|F|^3dx\Big)^\frac{2}{3}d\tau\\
&\le C\int_0^T\tau^{1-s}\Big(\intox|F|^2dx\Big)^\frac{1}{2}\Big(\intox|\nabla F|^2dx\Big)^\frac{1}{2}d\tau\\
&\le C\int_0^T\tau^{1-s}\Big(\intox(|\nabla u|^2+|P-\tilde P|^2)dx\Big)^\frac{1}{2}\Big(\intox(|\dot{u}|^2+|B|^2|\nabla B|^2)dx\Big)^\frac{1}{2}d\tau\\
&\le CC_0^\frac{1}{2}\Big(C_0+\int_0^T\tau^{1-s}\intox|B|^2|\nabla B|^2dxd\tau\Big)^\frac{1}{2}.
\end{align*}
The other terms on the right side of \eqref{bound on nabla u B} can be treated in a similar way, so we conclude that
\begin{align*}
&\intoxt \tau^{1-s}(|\nabla B|^2|B|^2+|\nabla u|^2|B|^2+|\nabla B|^2|u|^2)dxd\tau\\
&\qquad\le CC_0^\frac{1}{3}C_0^\frac{1}{2}\Big(C_0+\intoxt\tau^{1-s}(|\nabla B|^2|B|^2+|\nabla u|^2|B|^2+|\nabla B|^2|u|^2)dxd\tau\Big)^\frac{1}{2}
\end{align*}
and \eqref{auxiliary bound 1} follows.
\end{proof}

We now derive preliminary bounds for $\dot u$ and $B_t$ in $L^\infty([0,T];L^2(\R^3))$ when $T\le1$. We point out that here we require $s\in(\frac{1}{2},1]$ on the time layer factor $\tau^{2-s}$ due to the lack of integrability in time near $t=0$ for $\dot u$ and $B_t$.

\begin{lemma}\label{H2 estimates}
Assume that $\rho$ satisfies \eqref{pointwise bound on rho finite time} and $C_0\ll1$. For $T\le1$ and $s\in(\frac{1}{2},1]$, we have
\begin{align}\label{H2 bound}
\sup_{0\le \tau\le T}\tau^{2-s}\intox(|\dot{u}|^2+|B_t|^2)dx+\intoxt \tau^{2-s}(|\nabla\dot{u}|^2+|\nabla B_t|^2)dxd\tau\le CC_0^{\theta},
\end{align}
for some $\theta>0$.
\end{lemma}

\begin{proof}
Following the steps given in \cite{suenhoff12}, we arrive at
\begin{align}\label{H2 bound step 1}
&\sup_{0\le \tau\le T}\tau^{2-s}\intox(|\dot{u}|^2+|B_t|^2)dx+\intoxt \tau^{2-s}(|\nabla\dot{u}|^2+|\nabla B_t|^2)dxd\tau\notag\\
&\le CC_0+C\intoxt \tau^{2-s}|B|^2|u|^2(|\nabla u|^2+|\nabla B|^2)dxd\tau\notag\\
&\qquad+C\intoxt \tau^{2-s}(|B|^2|B_t|^2+|u|^2|\dot{u}|^2+|B_t|^2|u|^2)dxd\tau\notag\\
&\qquad+C\intoxt \tau^{2-s}(|\nabla u|^4+|\nabla B|^4)dxd\tau.
\end{align}
To facilitate the proof, we define the following auxiliary functionals:
\begin{align*}
\mathcal{A}_1(T)&=\sup_{0\le \tau\le T}\tau^{2-s}\intox(|\dot{u}|^2+|B_t|^2)dx+\intoxt \tau^{2-s}(|\nabla\dot{u}|^2+|\nabla B_t|^2)dxd\tau,\\
\mathcal{A}_2(T)&=\intoxt \tau^{2-s}(|\nabla u|^4+|\nabla B|^4)dxd\tau,\\
\mathcal{A}_3(T)&=\sup_{0\le \tau\le T}\intox \tau^{2-s}(|\nabla B|^2|B|^2+|\nabla u|^2|B|^2+|\nabla B|^2|u|^2)dx.
\end{align*}
Our goal is to prove that
\begin{align}\label{bound on A123}
\mathcal{A}_1+\mathcal{A}_2+\mathcal{A}_3\le CC_0^{\theta},
\end{align}
for some $\theta>0$. We first consider the right side of \eqref{H2 bound step 1}. To bound the term $\intoxt \tau^{2-s}|B|^2|B_t|^2dxd\tau$, using the bound \eqref{H1 bound}, we have
\begin{align*}
&\intoxt \tau^{2-s}|B|^2|B_t|^2dxd\tau\\
&\le\Big(\intoxt|B|^6dxd\tau\Big)^\frac{1}{3}\Big(\intoxt t^{\frac{6-3s}{2}}|B_t|^2dxd\tau\Big)^\frac{2}{3}\\
&\le CC_0^\frac{1}{3}\Big(\sup_{0\le \tau\le T}\tau^{2-s}\intox|B_t|^2dx\Big)^\frac{1}{3}\Big(\intoxt \tau^{2-s}|B_t|^2dxd\tau\Big)^\frac{1}{6}\\
&\qquad\times\Big(\intoxt \tau^{2-s}|\nabla B_t|^2dxd\tau\Big)^\frac{1}{2}\\
&\le CC_0^\frac{1}{3}\mathcal{A}_1^\frac{1}{3}C_0^\frac{1}{6}\mathcal{A}_1^\frac{1}{2}=CC_0^\frac{1}{2}\mathcal{A}_1^\frac{5}{6}.
\end{align*}
Similarly, the term $\intoxt \tau^{2-s}|B|^2|\dot{u}|^2$ can be bounded by
\begin{align*}
&\intoxt \tau^{2-s}|B|^2|\dot{u}|^2dxd\tau\\
&\le C\Big(\intoxt|B|^6dxd\tau\Big)^\frac{1}{3}\Big(\sup_{0\le \tau\le T}\tau^{2-s}\intox|\dot{u}|^2dx\Big)^\frac{1}{3}\\
&\qquad\times\Big(\intoxt \tau^{2-s}|\dot{u}|^2dxd\tau\Big)^\frac{1}{6}\Big(\intoxt \tau^{2-s}|\nabla\dot{u}|^2dxd\tau\Big)^\frac{1}{2}\\
&\le CC_0^\frac{1}{2}\mathcal{A}_1^\frac{5}{6}.
\end{align*}
To bound the term $\intoxt \tau^{2-s}|B|^2|u|^2(|\nabla u|^2+|\nabla B|^2)dxd\tau$, we have
\begin{align*}
&\intoxt \tau^{2-s}|B|^2|u|^2(|\nabla u|^2+|\nabla B|^2)dxd\tau\\
&\le \intoxt \tau^{2-s}(|\nabla u|^4+|\nabla B|^4)dxd\tau+\intoxt \tau^{2-s}(|B|^8+|u|^8)dxd\tau\\
&\le \mathcal{A}_2+\intoxt \tau^{2-s}(|B|^8+|u|^8)dxd\tau.
\end{align*}
Using the bounds \eqref{Lr bound general} and  \eqref{L infty bound general}, the term $\intoxt \tau^{2-s}|u|^8dxd\tau$ can be estimated as follows.
\begin{align*}
&\intoxt \tau^{2-s}|u|^8dxd\tau\\
&\le\Big(\sup_{0\le \tau\le T}\intox|u|^4dx\Big)\Big(\int_0^T\tau^{2-s}\|u(\cdot,\tau)\|^4_{L^\infty}d\tau\Big)\\
&\le C\Big(\sup_{0\le \tau\le T}\intox|u|^2dx\Big)^\frac{1}{2}\Big(\sup_{0\le \tau\le T}\intox|u|^6dx\Big)^\frac{1}{2}\\
&\qquad\times\int_0^T\tau^{2-s}\Big(\intox(|u|^4+|\nabla u|^4)dx\Big)d\tau\\
&\le CC_0\Big[\int_0^T\tau^{2-s}\Big(\intox|u|^2dx\Big)^\frac{1}{2}\Big(\intox|\nabla u|^2dx\Big)^\frac{3}{2}d\tau+\mathcal{A}_2\Big]\\
&\le CC_0\Big(C_0^2+\mathcal{A}_2\Big).
\end{align*}
The term $\intoxt \tau^{2-s}|B|^8dxd\tau$ can be estimated in a similar way and we obtain
\begin{align*}
\intoxt \tau^{2-s}|B|^2|u|^2(|\nabla u|^2+|\nabla B|^2)dxd\tau\le\mathcal{A}_2+CC_0\Big(C_0^2+\mathcal{A}_2\Big).
\end{align*}
Therefore we have
\begin{align}\label{bound on A1}
\mathcal{A}_1\le CC_0^\frac{1}{2}\mathcal{A}_1^\frac{5}{6}+CC_0\Big(C_0^2+\mathcal{A}_2\Big).
\end{align}
Next we consider the term $\mathcal{A}_2$. Using the estimate \eqref{bound on nabla u}, we obtain
\begin{align*}
\intoxt \tau^{2-s}|\nabla u|^4dxd\tau\le C\intoxt \tau^{2-s}(|F|^4+|\omega|^4+|P-\tilde P|^4)dxd\tau.
\end{align*}
For $s\in(\frac{1}{2},1]$, using the bounds \eqref{H1 bound} and \eqref{auxiliary bound 1}, we have 
\begin{align*}
&\intoxt \tau^{2-s}|F|^4dxd\tau\\
&\le C\Big(\sup_{0\le \tau\le T}\tau^{1-s}\intox|F|^2dx\Big)^\frac{1}{2}\\
&\qquad\qquad\times\Big(\sup_{0\le \tau\le T}\tau^{2-s}\intox|\nabla F|^2dx\Big)^\frac{1}{2}\Big(\intoxt \tau^{1-s}|\nabla F|^2dxd\tau\Big)\\
&\le CC_0^\frac{1}{2}\Big(\sup_{0\le \tau\le T}\tau^{2-s}\intox(|\dot{u}|^2+|\nabla B|^2|B|^2)dx\Big)^\frac{1}{2}\\
&\qquad\qquad\times\Big(\intoxt \tau^{1-s}(|\dot{u}|^2+|\nabla B|^2|B|^2)dxd\tau\Big)^\frac{1}{2}\\
&\le CC_0^\frac{1}{2}\Big(\mathcal{A}_1+\mathcal{A}_3\Big)\Big(C_0+C_0^\theta\Big).
\end{align*}
The term $\omega$ and $P-\tilde P$ can be treated in a similar way, and we have
\begin{align*}
\intoxt \tau^{2-s}|\nabla u|^4dxd\tau\le CC_0^\frac{1}{2}\Big(\mathcal{A}_1+\mathcal{A}_3\Big)\Big(C_0+C_0^\theta\Big)+CC_0.
\end{align*}
For $\intoxt \tau^{2-s}|\nabla B|^4dxd\tau$, we estimate it as follows.
\begin{align*}
&\intoxt \tau^{2-s}|\nabla B|^4dxd\tau\\
&\le C\int_0^T\tau^{2-s}\Big(\intox|\nabla B|^2dx\Big)^\frac{1}{2}\Big(\intox|\Delta B|^2dx\Big)^\frac{3}{2}d\tau\\
&\le C\Big(\sup_{0\le \tau\le T}\tau^{2-s}\intox|\Delta B|^2dx\Big)^\frac{1}{2}\Big(\sup_{0\le \tau\le T}\tau^{1-s}\intox|\nabla B|^2dx\Big)^\frac{1}{2}\\
&\qquad\times\Big(\intoxt \tau^{1-s}|\Delta B|^2dxd\tau\Big)\\
&\le C\Big(\mathcal{A}_1+\mathcal{A}_3\Big)^\frac{1}{2}C_0^\frac{1}{2}\Big(C_0+C_0^\theta\Big).
\end{align*}
Hence we obtain
\begin{align}\label{bound on A2}
\mathcal{A}_2\le CC_0^\frac{1}{2}\Big(\mathcal{A}_1+\mathcal{A}_3\Big)\Big(C_0+C_0^\theta\Big)+CC_0+C\Big(\mathcal{A}_1+\mathcal{A}_3\Big)^\frac{1}{2}C_0^\frac{1}{2}\Big(C_0+C_0^\theta\Big).
\end{align}
It remains to estimate $\mathcal{A}_3$. For $s\in(\frac{1}{2},1]$, we define $r=\frac{6}{3-2s}$, then we have $r\in(3,6]$. Simple computations yield
\begin{align*}
\frac{6-r}{2r}=1-s,\qquad\frac{3r-6}{2r}=s.
\end{align*}
We bound the term $\tau^{2-s}\intox|\nabla u|^2|B|^2dx$ as follows. Since $r>3$, we can apply \eqref{L infty bound general} to obtain
\begin{align*}
&\tau^{2-s}\intox|\nabla u|^2|B|^2dx\\
&\le \tau\Big(\tau^{1-s}\intox|\nabla u|^2dx\Big)\|B(\cdot,\tau)\|^2_{L^\infty}\\
&\le C\tau\Big(\tau^{1-s}\intox|\nabla u|^2dx\Big)\Big(\intox|B|^rdx+\intox|\nabla B|^rdx\Big)^\frac{2}{r}.
\end{align*}
The term $\|B\|_{L^r}$ can be bounded by the $L^2-L^6$ interpolation on $B$. Hence using \eqref{Lr bound general} and the bound \eqref{H1 estimates}, we obtain
\begin{align*}
&\tau^{2-s}\intox|\nabla u|^2|B|^2dx\\
&\le C\tau C_0^3+C\tau\Big(\tau^{1-s}\intox|\nabla u|^2dx\Big)\Big(\intox|\nabla B|^2dx\Big)^\frac{6-r}{2r}\Big(\intox|\Delta B|^2dx\Big)^\frac{3r-6}{2r}\\
&\le C\tau C_0^3+C\Big(\tau^{1-s}\intox|\nabla u|^2dx\Big)\Big(\tau^{1-s}\intox|\nabla B|^2dx\Big)^{1-s}\Big(\tau^{2-s}\intox|\Delta B|^2dx\Big)^s\\
&\le C\tau C_0^3+CC_0C_0^{1-s}\Big(\tau^{2-s}\intox|B_t|^2dx+\intox(|\nabla B|^2|u|^2+|\nabla u|^2|B|^2)dx\Big)^s\\
&\le C\tau C_0^3+CC_0C_0^{1-s}\Big(\mathcal{A}_1+\mathcal{A}_3\Big)^s.
\end{align*}
By similar method, we obtain
\begin{align*}
\tau^{2-s}\intox|\nabla B|^2|B|^2dx&\le t\Big(\tau^{1-s}\intox|\nabla B|^2dx\Big)\|B(\cdot,\tau)\|^2_{L^\infty}\\
&\le C\tau C_0^3+CC_0C_0^{1-s}\Big(\mathcal{A}_1+\mathcal{A}_3\Big)^s.
\end{align*}
For the term $\tau^{2-s}\intox|\nabla B|^2|u|^2dx$, we have
\begin{align*}
\tau^{2-s}\intox|\nabla B|^2|u|^2dx&\le \tau\Big(\tau^{1-s}\intox|\nabla B|^2dx\Big)\|u(\cdot,\tau)\|^2_{L^\infty}\\
&\le C\tau\Big(\tau^{1-s}\intox|\nabla B|^2dx\Big)\Big(\intox|u|^rdx+\intox|\nabla u|^rdx\Big)^\frac{2}{r}.
\end{align*}
Similar to the case for $B$, the term $\|u\|_{L^r}$ can be bounded by the $L^2-L^6$ interpolation on $u$. For the term $\|\nabla u\|_{L^r}$, we can apply \eqref{bound on nabla u} with the bounds \eqref{bound on F}-\eqref{bound on nabla F} to get
\begin{align*}
\tau\Big(\intox|\nabla u|^rdx\Big)^\frac{2}{r}&\le C\tau\Big(\intox(|F|^r+|\omega|^r+|\rho-\tilde\rho|^r)dx\Big)^\frac{2}{r}\\
&\le C\Big(\tau^{1-s}\intox|F|^2dx\Big)^{1-s}\Big(\tau^{1-s}\intox|\nabla F|^2dx\Big)^{s}\\
&+C\Big(\tau^{1-s}\intox|\omega|^2dx\Big)^{1-s}\Big(\tau^{1-s}\intox|\nabla \omega|^2dx\Big)^{s}+C\Big(\intox|\rho-\tilde\rho|^rdx\Big)^\frac{2}{r}\\
&\le CC_0^{1-s}\Big(\mathcal{A}_1+\mathcal{A}_3\Big)^s+C\Big(\intox|\rho-\tilde\rho|^rdx\Big)^\frac{2}{r}.
\end{align*}
Together with the bound on $P-\tilde P$ given by \eqref{L2 bound}, we obtain
\begin{align*}
&\tau^{2-s}\intox|\nabla B|^2|u|^2dx\\
&\le C\tau C_0^3+CC_0C_0^{1-s}\Big(\mathcal{A}_1+\mathcal{A}_3\Big)^s+CC_0C_0^\frac{2}{r},
\end{align*}
and therefore by taking supremum over $\tau\in[0,T]$,
\begin{align}\label{bound on A3}
\mathcal{A}_3\le CC_0^3+CC_0C_0^{1-s}\Big(\mathcal{A}_1+\mathcal{A}_3\Big)^s+CC_0C_0^\frac{2}{r}.
\end{align}
Combining the results \eqref{bound on A1}, \eqref{bound on A2} and \eqref{bound on A3}, and utilising the smallness condition on $C_0$, there is some $\theta>0$ such that \eqref{bound on A123} holds. In particular, 
\begin{align*}
\sup_{0\le \tau\le T}\tau^{2-s}\intox(|\dot{u}|^2+|B_t|^2)dx+\intoxt \tau^{2-s}(|\nabla\dot{u}|^2+|\nabla B_t|^2)dxd\tau\le CC_0^{\theta}.
\end{align*}
We finish the proof of \eqref{H2 bound}.
\end{proof}

By combining the estimates obtained in previous lemmas, we have the following estimates on $\mathcal{A}(T)$ when $T\le1$ under the pointwise bound \eqref{pointwise bound on rho finite time} on $\rho$ and smallness assumption on $C_0$. In other words, for $C_0\ll1$ and $T\le1$, we have
\begin{equation}\label{bound on A(t) t<1}
\mathcal{A}(T)\le CC_0^{\theta},
\end{equation}
for some $\theta>0$.

\subsection{Estimates on $u$ and $B$ for $T>1$}\label{large time estimates} 

In this subsection, we proceed to estimate $\mathcal{A}(T)$ when $T>1$. In view of the definition of $\mathcal{A}$ and the bounds \eqref{L2 bound} and \eqref{bound on A(t) t<1}, it suffices to bound the following term for $t>1$:
\begin{align}\label{def of A(t) t>1}
\cA (t)&=\sup_{1\le\tau\le t}\Big[\intox(|\nabla u|^2+|\nabla B|^2+|\dot{u}|^2+|B_t|^2)\Big](x,\tau)dx\notag\\
&\qquad+\intoxtl(|\dot{u}|^2+|B_t|^2+|\nabla\dot{u}|^2+|\nabla B_t|^2)dxd\tau.
\end{align}

Before we estimate $\cA (t)$, we state the following estimates on $\rho$ which was proved in Hoff \cite{hoff95}.

\begin{proposition}
Assume that $\rho$ satisfies \eqref{pointwise bound on rho finite time}, then for $t>1$, we have
\begin{equation}\label{bound on L4 of rho T>1}
\intoxtl |\rho-\tilde\rho|^4dxd\tau\le C\intoxtl  |F|^4dxd\tau+CC_0.
\end{equation}
\end{proposition}

We further introduce the following auxiliary bounds \eqref{L4 norm on u and B T>1}-\eqref{L infty norm on u and B T>1} which will be used in controlling $\cA(t)$. For simplicity, we define an auxiliary functional $\cH(t)$ by
\[\cH(t)=\intoxtl (|\nabla u|^4+|\nabla B|^4)dxd\tau.\]

\begin{lemma}\label{auxiliary estimates T>1}
Assume that $\rho$ satisfies \eqref{pointwise bound on rho finite time}, then for $t>1$, we have
\begin{align}
&\intoxtl(|u|^4+|B|^4)dxd\tau\le CC_0^\frac{3}{2}\cA (t)^\frac{1}{2},\label{L4 norm on u and B T>1}\\
&\intoxtl(|u|^8+|B|^8)dxd\tau\le CC_0^2\cA (t)^2+CC_0^\frac{1}{2}\cA (t)^\frac{3}{2}\cH(t),\label{L8 norm on u and B T>1}\\
&\sup_{1\le \tau\le t}\Big(\|u(\cdot,\tau)\|_{L^\infty}+\|B(\cdot,\tau)\|_{L^\infty}\Big)\le C\Big(\cA (t)^2+\cA (t)^\frac{1}{2}+C_0^{\theta}\Big),\label{L infty norm on u and B T>1}
\end{align}
where $\theta>0$.
\end{lemma}

\begin{proof}
To prove \eqref{L4 norm on u and B T>1}, using \eqref{Lr bound general}, \eqref{L2 bound} and the definition of $\cA(t)$, we can bound the term $\intoxtl(|u|^4+|B|^4)$ as follows.
\begin{align*}
&\intoxtl(|u|^4+|B|^4)dxd\tau\notag\\
&\le C\int_1^T\Big(\intox|u|^2dx\Big)^\frac{1}{2}\Big(\intox|\nabla u|^2dx\Big)^\frac{3}{2}d\tau+C\int_1^T\Big(\intox|B|^2dx\Big)^\frac{1}{2}\Big(\intox|\nabla B|^2dx\Big)^\frac{3}{2}d\tau\notag\\
&\le CC_0^\frac{3}{2}\cA ^\frac{1}{2}.
\end{align*} 
Next to prove \eqref{L8 norm on u and B T>1}, using \eqref{L infty bound general} and \eqref{L4 norm on u and B T>1}, we have
\begin{align*}
&\intoxtl|u|^8dxd\tau\\
&\le \Big(\sup_{1\le \tau\le t}\intox|u|^4dx\Big)\int_1^T\|u(\cdot,\tau)\|^4_{L^\infty}d\tau\\
&\le C \Big(\sup_{1\le \tau\le t}\intox|u|^2dx\Big)^\frac{1}{2}\Big(\sup_{1\le \tau\le t}\intox|\nabla u|^2dx\Big)^\frac{3}{2}\Big(\intoxtl(|u|^4+|\nabla u|^4)dxd\tau\Big)\\
&\le CC_0^2\cA ^2+CC_0^\frac{1}{2}\cA ^\frac{3}{2}\intoxtl|\nabla u|^4dxd\tau,
\end{align*}
and similarly,
\begin{align*}
\intoxtl|B|^8dxd\tau\le CC_0^2\cA ^2+CC_0^\frac{1}{2}\cA ^\frac{3}{2}\intoxtl|\nabla B|^4dxd\tau.
\end{align*}
Finally, to show \eqref{L infty norm on u and B T>1}, we use \eqref{L infty bound general}, \eqref{bound on nabla u} and \eqref{L2 bound} to obtain, for $1\le\tau\le t$,
\begin{align*}
&\|u(\cdot,\tau)\|_{L^\infty}\\
&\le C\Big(\intox|u|^4dx\Big)^\frac{1}{4}+C\Big(\intox|\nabla u|^4dx\Big)^\frac{1}{4}\\
&\le C\Big(\intox|u|^2dx\Big)^\frac{1}{8}\Big(\intox|\nabla u|^2dx\Big)^\frac{3}{8}\\
&\qquad+C\Big[\Big(\intox|F|^4dx\Big)^\frac{1}{4}+\Big(\intox|\omega|^4dx\Big)^\frac{1}{4}+\Big(\intox|\rho-\tilde\rho|^4dx\Big)^\frac{1}{4}\Big]\\
&\le CC_0^\frac{1}{8}\cA^\frac{3}{8}+C\Big[\Big(\intox|F|^4dx\Big)^\frac{1}{4}+\Big(\intox|\omega|^4dx\Big)^\frac{1}{4}+C_0^\frac{1}{4}\Big].
\end{align*}
To estimate the term involving $F$, we apply \eqref{bound on F} and \eqref{bound on F} together with \eqref{Lr bound general} to have
\begin{align*}
&\Big(\intox|F|^4dx\Big)^\frac{1}{4}\\
&\le C\Big(\intox|F|^2dx\Big)^\frac{1}{8}\Big(\intox|\nabla F|^2dx\Big)^\frac{3}{8}\\
&\le C\Big(\intox(|\nabla u|^2+|\rho-\tilde\rho|^2)dx\Big)^\frac{1}{8}\Big(\intox(|\dot{u}|^2+|\nabla B|^2|B|^2)dx\Big)^\frac{3}{8}\\
&\le C\Big(\cA+C_0\Big)^\frac{1}{8}\Big(\cA+\|B(\cdot,\tau)\|_{L^\infty}^2\cA\Big)^\frac{3}{8}.
\end{align*}
The term involving $\omega$ can be estimated in a similar way to $F$, so we obtain
\begin{align*}
\|u(\cdot,\tau)\|_{L^\infty}\le C\Big(C_0^\frac{1}{8}\cA^\frac{3}{8}+\cA^\frac{1}{2}+\cA^\frac{1}{2}\|B(\cdot,\tau)\|_{L^\infty}^\frac{3}{4}+C_0^\frac{1}{8}\cA^\frac{3}{8}\|B(\cdot,\tau)\|_{L^\infty}^\frac{3}{4}+C_0^\frac{1}{4}\Big).
\end{align*}
For $\|B(\cdot,\tau)\|_{L^\infty}$, we can estimate it as follows.
\begin{align*}
&\|B(\cdot,\tau)\|_{L^\infty}\\
&\le C\Big(\intox|B|^4dx\Big)^\frac{1}{4}+C\Big(\intox|\nabla B|^4dx\Big)^\frac{1}{4}\\
&\le C\Big(\intox|B|^2dx\Big)^\frac{1}{8}\Big(\intox|\nabla B|^2dx\Big)^\frac{3}{8}+C\Big(\intox|\nabla B|^2dx\Big)^\frac{1}{8}\Big(\intox|\Delta B|^2dx\Big)^\frac{3}{8}\\
&\le CC_0^\frac{1}{8}\cA^\frac{3}{8}+C\cA^\frac{1}{8}\Big(\intox(|B_t|^2+|\nabla B|^2|u|^2+|\nabla u|^2|B|^2)dx\Big)^\frac{3}{8}\\
&\le CC_0^\frac{1}{8}\cA^\frac{3}{8}+C\cA^\frac{1}{8}\Big(\cA+(\|B\|_{L^\infty}+\|u\|_{L^\infty})^2\cA\Big)^\frac{3}{8}\\
&\le CC_0^\frac{1}{8}\cA^\frac{3}{8}+C\cA^\frac{1}{2}+C\cA^\frac{1}{2}(\|B\|_{L^\infty}+\|u\|_{L^\infty})^\frac{3}{4}.
\end{align*}
Therefore with the help of Cauchy's inequality, the bound \eqref{L infty norm on u and B T>1} then follows.
\end{proof}

We are now going to obtain the estimates on $\cA (t)$. The main difference between here and the analysis given in subsection~\ref{small time estimates} is that, we {\it cannot} use the $L^6$ estimates on $u$ and $B$ in controlling $\cA (t)$ for $t>1$. We therefore need to make use of some new methods which will be illustrated in Lemma~\ref{estimates on A(t) t>1 lemma} as below.

\begin{lemma}\label{estimates on A(t) t>1 lemma}
Assume that $\rho$ satisfies \eqref{pointwise bound on rho finite time}. For $C_0\ll1$ and $t>1$, we have
\begin{equation}\label{bound on A(t) t>1}
\cA (t)\le CC_0^{\theta},
\end{equation}
for some $\theta>0$.
\end{lemma}

\begin{proof}
Following the similar argument given in the proof of Lemma~\ref{H1 estimates}, we have for $t\in[1,T]$ that
\begin{align*}
&\sup_{1\le\tau\le t}\intox(|\nabla u|^2+|\nabla B|^2)(x,\tau)dx+\intoxtl(|\dot{u}|^2+|B_t|^2)dxd\tau\\
&\le \mathcal{A}(1)+\intoxtl(|\nabla B|^2|B|^2+|\nabla u|^2|B|^2+|\nabla B|^2|u|^2)dxd\tau.
\end{align*}
Hence using \eqref{L4 norm on u and B T>1}, we obtain
\begin{align}\label{bound on auxiliary term T>1}
&\intoxtl(|\nabla B|^2|B|^2+|\nabla u|^2|B|^2+|\nabla B|^2|u|^2)dxd\tau\notag\\
&\le C\Big[\Big(\intoxtl|u|^4dxd\tau\Big)^\frac{1}{2}+\Big(\intoxtl|B|^4dxd\tau\Big)^\frac{1}{2}\Big]\notag\\
&\qquad\times\Big[\Big(\intoxtl|\nabla u|^4dxd\tau\Big)^\frac{1}{2}+\Big(\intoxtl|\nabla B|^4dxd\tau\Big)^\frac{1}{2}\Big]\notag\\
&\le CC_0^\frac{3}{4}\cA ^\frac{1}{4}\cH^\frac{1}{2},
\end{align}
and thus
\begin{align}\label{H1 estimate T>1}
&\sup_{1\le\tau\le t}\intox(|\nabla u|^2+|\nabla B|^2)(x,\tau)dx+\intoxtl(|\dot{u}|^2+|B_t|^2)dxd\tau\notag\\
&\le \mathcal{A}(1)+CC_0^\frac{3}{4}\cA ^\frac{1}{4}\cH^\frac{1}{2}.
\end{align}
On the other hand, following the argument given in the proof of Lemma~\ref{H2 estimates}, we have,  for $\in[1,T]$ that
\begin{align*}
&\sup_{1\le\tau\le t}\intox(|\dot{u}|^2+|B_t|^2)(x,\tau)dx+\intoxtl (|\nabla\dot{u}|^2+|\nabla B_t|^2)dxd\tau\notag\\
&\le \mathcal{A}(1)+C\intoxtl |B|^2|u|^2(|\nabla u|^2+|\nabla B|^2)dxd\tau\notag\\
&\qquad+C\intoxtl (|B|^2|B_t|^2+|u|^2|\dot{u}|^2+|B_t|^2|u|^2)dxd\tau\notag\\
&\qquad+C\intoxtl (|\nabla u|^4+|\nabla B|^4)dxd\tau.
\end{align*}
The term $\intoxtl |B|^2|u|^2(|\nabla u|^2+|\nabla B|^2)dxd\tau$ can be estimated as follows.
\begin{align*}
&\intoxtl |B|^2|u|^2(|\nabla u|^2+|\nabla B|^2)dxd\tau\\
&\le \intoxtl(|u|^8+|B|^8)dxd\tau+\intoxtl(|\nabla u|^4+|\nabla B|^4)dxd\tau.
\end{align*}
The term $\intoxtl(|u|^8+|B|^8)dxd\tau$ can be bounded using \eqref{L8 norm on u and B T>1}, and to bound the term $\intoxtl |B|^2|B_t|^2dxd\tau$, we can use \eqref{Lr bound general} and \eqref{L2 bound} to get
\begin{align*}
&\intoxtl |B|^2|B_t|^2dxd\tau\\
&\le \int_1^T\Big(\intox|B|^2dx\Big)^\frac{1}{4}\Big(\intox|\nabla B|^2dx\Big)^\frac{3}{4}\Big(\intox|B_t|^2dx\Big)^\frac{1}{4}\Big(\intox|\nabla B_t|^2dx\Big)^\frac{3}{4}d\tau\\
&\le CC_0^\frac{1}{2}\cA ^\frac{3}{2}.
\end{align*}
The terms $\intoxtl |B_t|^2|u|^2dxd\tau$ and $\intoxtl |B_t|^2|u|^2dxd\tau$ can be treated in a similar way which gives 
\begin{equation*}
\intoxtl (|B|^2|B_t|^2+|u|^2|\dot{u}|^2+|B_t|^2|u|^2)dxd\tau\le CC_0^\frac{1}{2}\cA ^\frac{3}{2}.
\end{equation*}
Therefore we have
\begin{align}\label{H2 estimate T>1}
&\sup_{1\le\tau\le t}\intox(|\dot{u}|^2+|B_t|^2)(x,\tau)dx+\intoxtl (|\nabla\dot{u}|^2+|\nabla B_t|^2)dxd\tau\notag\\
&\le \mathcal{A}(1)+CC_0^2\cA ^2+CC_0^\frac{1}{2}\cA ^\frac{3}{2}\cH+CC_0^\frac{1}{2}\cA ^\frac{3}{2}+C\cH.
\end{align}
It remains to estimate $\cH$. Using \eqref{L infty bound general}, \eqref{bound on nabla u}, \eqref{L2 bound} and \eqref{bound on L4 of rho T>1}, we have
\begin{align*}
\intoxtl|\nabla u|^4dxd\tau\le C\Big(\intoxtl|F|^4dxd\tau+\intoxtl|\omega|^4dxd\tau\Big)+CC_0.
\end{align*}
To control $\intoxtl|F|^4dxd\tau$, we use \eqref{Lr bound general} to get
\begin{align*}
&\intoxtl|F|^4dxd\tau\\
&\le C\int_1^T\Big(\intox|F|^2dx\Big)^\frac{1}{2}\Big(\intox|\nabla F|^2dx\Big)^\frac{3}{2}d\tau\\
&\le C\Big(\sup_{1\le \tau\le t}\int|\nabla F|^2dx\Big)^\frac{1}{2}\Big(\sup_{1\le \tau\le t}\int|F|^2dx\Big)^\frac{1}{2}\Big(\intoxtl|\nabla F|^2dxd\tau\Big)^\frac{1}{2}.
\end{align*}
The term $\Big(\sup_{1\le \tau\le t}\int|F|^2dx\Big)^\frac{1}{2}$ can be bounded by $C(\cA+C_0)^\frac{1}{2}$, and using \eqref{bound on auxiliary term T>1}, the term $\Big(\intoxtl|\nabla F|^2dxd\tau\Big)^\frac{1}{2}$ can be bound by
\begin{align*}
C\Big(\intoxtl|\dot{u}|^2dxd\tau+\intoxtl|\nabla B|^2|B|^2dxd\tau\Big)^\frac{1}{2}\le C\Big(\cA^\frac{1}{2}+C_0^\frac{3}{8}\cA^\frac{1}{8}\cH^\frac{1}{4}\Big).
\end{align*}
For $\Big(\sup_{1\le \tau\le t}\int|\nabla F|^2dx\Big)^\frac{1}{2}$, we use \eqref{L infty norm on u and B T>1} to otbain
\begin{align*}
\Big(\sup_{1\le \tau\le t}\int|\nabla F|^2dx\Big)^\frac{1}{2}&\le C\Big(\sup_{1\le \tau\le t}\int|\dot{u}|^2dx+\sup_{1\le \tau\le t}\int|\nabla B|^2|B|^2dx\Big)^\frac{1}{2}\\
&\le C\Big(\cA+\sup_{1\le \tau\le t}\|B(\cdot,\tau)\|^2_{L^\infty}\intox|\nabla B|^2dx\Big)^\frac{1}{2}\\
&\le C\Big(\cA+(\cA (t)^2+\cA (t)^\frac{1}{2}+C_0^{\theta})^2\cA\Big)^\frac{1}{2}
\end{align*}
The term $\intoxtl|\omega|^4dxd\tau$ can be estimated in a similar way, and we have
\begin{align*}
&\intoxtl|\nabla u|^4dxd\tau\\
&\le CC_0+C(\cA+C_0)^\frac{1}{2}\Big(\cA+(\cA (t)^2+\cA (t)^\frac{1}{2}+C_0^{\theta})^2\cA\Big)^\frac{1}{2}\Big(\cA^\frac{1}{2}+C_0^\frac{3}{8}\cA^\frac{1}{8}\cH^\frac{1}{4}\Big).
\end{align*}
We estimate $\intoxtl|\nabla u|^4$ in a similar way as the case for $\nabla u$ and get
\begin{align*}
&\intoxtl|\nabla B|^4dxd\tau\\
&\le C\int_1^T\Big(\intox|\nabla B|^2dx\Big)^\frac{1}{2}\Big(\intox|\Delta B|^2dx\Big)^\frac{3}{2}d\tau\\
&\le C\int_1^T\Big(\intox|\nabla B|^2dx\Big)^\frac{1}{2}\Big(\intox(|B_t|^2+|B|^2|\nabla u|^2+|\nabla B|^2|u|^2)dx\Big)^\frac{3}{2}d\tau\\
&\le C(\cA+C_0)^\frac{1}{2}\Big(\cA+(\cA (t)^2+\cA (t)^\frac{1}{2}+C_0^{\theta})^2\cA\Big)^\frac{1}{2}\Big(\cA^\frac{1}{2}+C_0^\frac{3}{8}\cA^\frac{1}{8}\cH^\frac{1}{4}\Big).
\end{align*}
Using Cauchy's inequality, there are constants $\bar{\theta}>1$ and $\theta>0$ such that
\begin{equation}\label{estimate on H T>1}
\cH\le C\Big(\cA^{\bar{\theta}}+C_0^{\theta}\Big).
\end{equation}
By combining \eqref{H1 estimate T>1}, \eqref{H2 estimate T>1}, \eqref{estimate on H T>1}, replacing $\bar{\theta}>1$ and $\theta>0$ if necessary, we conclude that, for all $t\in[1,T]$ that
\begin{equation*}
\cA(t)\le C\Big(\cA(t)^{\bar{\theta}}+C_0^{\theta}\Big)+C\mathcal{A}(1).
\end{equation*}
Using the bound \eqref{bound on A(t) t<1} on $\mathcal{A}(1)$ and making use of the smallness assumption on $C_0$, the result \eqref{bound on A(t) t>1} then follows.
\end{proof}

Combining \eqref{bound on A(t) t<1} and \eqref{bound on A(t) t>1}, under the assumption \eqref{pointwise bound on rho finite time} and $C_0\ll1$, we have for all $T>0$ that
\begin{equation}\label{bound on A(t) all t}
\mathcal{A}(T)\le CC_0^{\theta}.
\end{equation}

\subsection{Pointwise bound on $\rho$ and proof of Theorem~\ref{a priori bounds finite time theorem}}\label{Pointwise bound on rho subsection}

We now close the estimates on $\mathcal{A}$ by proving the following pointwise bounds on $\rho$. Together with \eqref{bound on A(t) all t}, it will finish the proof of Theorem~\ref{a priori bounds finite time theorem}.

\begin{lemma}\label{pointwise bound on density}
For $C_0\ll1$, we have the pointwise bounds for $\rho$, namely
\begin{align}\label{pointwise bound on rho}
\frac{1}{2}\rho_1\le\rho(x,t)\le 2\rho_2,
\end{align}
for all $t\in[0,T]$ and $x\in\R^3$.
\end{lemma}

\begin{proof}
The proof is reminiscent of the one given in \cite{suenhoff12} which consists of a maximum-principle argument applied along integral curves of the velocity field $u$, and we only give the sketch here. First we fix $y\in\R^3$ and define the corresponding integral curve of $u$ by
\begin{align}\label{integral curve of u}
\left\{ \begin{array}
{lr} \dot{x}(t)
=u(x(t),t)\\ x(0)=y.
\end{array} \right.
\end{align}
From the definition \eqref{definition of F} of $F$ and the mass equation \eqref{MHD1} that
\begin{align*}
(\mu+\lambda)\frac{d}{dt}[\log\rho(x(t),t)-\log(\tilde\rho)]+P(\rho(x(t),t))-\tilde P=-F(x(t),t).
\end{align*}
Integrating from $t_0$ to $t_1$ for $t_1,t_2\in[0,T]$, and abbreviating $\rho(x(t),t)$ by $\rho(t)$, etc., we then obtain
\begin{align}\label{integral on F}
(\mu+\lambda)[\log\rho(s)-\log(\tilde\rho)]\Big |_{ t_0}^{t_1}+\int_{ t_0}^{t_1}[P(\tau)-\tilde P]d\tau=-\int_{ t_0}^{t_1}F(\tau)d\tau.
\end{align}
Since $P$ is increasing, the integral of $P$ on the left side of \eqref{integral on F} is thus a dissipative term. On the other hand, if we can prove that
\begin{align}\label{smallness on integral of F}
\int_{ t_0}^{t_1}F(\tau)d\tau\le CC_0^{\theta'},
\end{align}
for some $\theta'>0$, then by stipulating the smallness condition on $C_0$, we can see that the density $\rho$ should remain inside the interval $[\frac{1}{2}\rho_1,2\rho_2]$ for all $t\in[0,T]$, provided that the initial density satisfies $\rho_0(x)\in[\rho_1,\rho_2]$ for $x\in\R^3$. To see why \eqref{smallness on integral of F} holds, using the Poisson equation \eqref{poisson in F}, we can rewrite the integral of $F$ as 
\begin{align*}
\int_{ t_0}^{t_1}F(\tau)d\tau&=\int_{ t_0}^{t_1}(\Gamma_{x_j}*g^j)(\cdot,\tau)(x(\tau))d\tau\notag\\
&=\int_{ t_0}^{t_1}\!\!\!\int_{\R^3}\Gamma_{x_j}(x(\tau)-y)\rho\dot{u}^{j}(y,\tau)dyd\tau\notag\\
&\qquad-\int_{ t_0}^{t_1}\!\!\!\int_{\R^3}\Gamma_{x_j}(x(\tau)-y)\left[\divv(BB^{j}(y,\tau))-B\cdot B_{x_j}(y,\tau)\right]dyd\tau.
\end{align*}
The integrals on the right here can be bounded using the estimates on the H\"{o}lder's norm of $u$ (which are given in Lemma~\ref{Holder norm estimates}) and a time integral on $\|B\cdot\nabla B(\cdot,t)\|_{L^4}$, and hence \eqref{smallness on integral of F} holds. The argument can be made rigorous and we refer to \cite{suenhoff12} for more details.
\end{proof}

\section{Proof of Theorem~\ref{Existence theorem}}\label{proof of existence section}

In this section, we give the proof of Theorem~\ref{Existence theorem}. We fix the constants $d$ and $C$ defined in Theorem~\ref{a priori bounds finite time theorem}. We let initial data $(\rho_0,u_0,B_0)$ be given satisfying the hypotheses \eqref{bound on initial density}-\eqref{definition of C0} and take $(\rho^\varepsilon_0,u^\varepsilon_0,B^\varepsilon_0)$ to be smooth approximate initial data obtained by convolving $(\rho_0,u_0,B_0)$ with a standard mollifying kernel of width $\varepsilon>0$. Then by Theorem~\ref{local existence theorem} as described in Section~\ref{a priori estimates}, for each $\varepsilon$, under the smallness assumption \eqref{definition of C0} on $C_0$ there is a corresponding smooth local solution $(\rho^\varepsilon,u^\varepsilon,B^\varepsilon)$ satisfying the bounds \eqref{bound on weak solution finite time}-\eqref{pointwise bound on rho finite time}. 

We recall the following theorem which shows that the smooth local solution $(\rho^\varepsilon,u^\varepsilon,B^\varepsilon)$ as described above exists on all of $\R^3\times [0,\infty)$. The proof can be found in \cite{suenhoff12} pp. 51--56.

\begin{theorem}\label{global smooth existence theorem}
Assume that the system parameters in \eqref{MHD1}-\eqref{MHD4} satisfy the conditions in \eqref{condition on vis}-\eqref{condition on pressure}, and let $d$, $\theta$, $C$ be as described in Theorem~\ref{Existence theorem}. Then given initial data $(\rho_0-\tilde\rho,u_0,B_0)\in H^3(\R^3)$ satisfying \eqref{bound on initial density}-\eqref{definition of C0}, the corresponding smooth solution $(\rho,u,B)$ described in Theorem~\ref{local existence theorem} exists on all of $\R^3\times [0,\infty)$.
\end{theorem}

Using the estimates \eqref{bound on weak solution finite time}-\eqref{pointwise bound on rho finite time} from Theorem~\ref{a priori bounds finite time theorem}, we obtain bounds on $(\rho^\varepsilon,u^\varepsilon,B^\varepsilon)$ which will provide the compactness needed to extract the desired global-in-time weak solution $(\rho,u,B)$ in the limit as $\varepsilon\to 0$. We first derive a result on the H\"{o}lder-continuity of $u^\varepsilon(\cdot,t)$, $B^\varepsilon(\cdot,t)$, $F^\varepsilon(\cdot,t)$ and $\omega^\varepsilon(\cdot,t)$ (here $F^\varepsilon$ and $\omega^\varepsilon$ are defined in terms of $\rho^\varepsilon$, $u^\varepsilon$ and $B^\varepsilon$).

\begin{lemma}\label{Holder norm estimates}
For $\alpha\in(0,\frac{1}{2}]$ and $t\in(0,\infty)$, we have the following estimates on the H\"{o}lder's norms on $u^\varepsilon$, $B^\varepsilon$, $F^\varepsilon$ and $\omega^\varepsilon$: for any $\varepsilon>0$, we have
\begin{align}\label{Holder norm}
\langle u^\varepsilon(\cdot,t)\rangle^{\alpha}+\langle B^\varepsilon(\cdot,t)\rangle^{\alpha}+\langle F^\varepsilon(\cdot,t)\rangle^{\alpha}+\langle \omega^\varepsilon(\cdot,t)\rangle^{\alpha}\le C(t)C_0^{\theta},
\end{align}
for some $\theta>0$ and $C(t)>0$ may depend on $t$ but independent of $\varepsilon$.
\end{lemma}

\begin{proof}
We only give the proof for $u^\varepsilon$ and the others are just similar. By the estimates \eqref{holder bound general} and \eqref{bound on nabla u},
\begin{align*}
\langle u^\varepsilon(\cdot,t)\rangle^{\alpha}\le C\left[||F^\varepsilon(\cdot,t)||_{L^r}+||\omega^\varepsilon(\cdot,t)||_{L^r}+||(\rho^\varepsilon-\tilde\rho)(\cdot,t)||_{L^r}\right].
\end{align*}
On the other hand, by \eqref{Lr bound general}, we have
\begin{align*}
\|F^\varepsilon(\cdot,t)\|_{L^r}&\le C\left(\|F^\varepsilon(\cdot,t)\|_{L^2}^{(6-r)/2r}\|\nabla F^\varepsilon(\cdot,t)\|_{L^2}^{(3r-6)/2r}\right)\\
&\le C\left(||(\rho^\varepsilon-\tilde\rho)(\cdot,t)||^2_{L^2}+||\nabla u^\varepsilon(\cdot,t)||^2_{L^2}\right)^\frac{1-2\alpha}{4}\\
&\qquad\times\left(||\dot{u}^\varepsilon(\cdot,t)||^2_{L^2}+||\nabla B^\varepsilon\cdot B^\varepsilon(\cdot,t)||^2_{L^2}\right)^\frac{1+2\alpha}{4}.
\end{align*}
The desired result then follows by applying the estimates \eqref{bound on weak solution finite time}-\eqref{pointwise bound on rho finite time}.
\end{proof}

Next we recall the following result on the H\"{o}lder-continuity of $u^\varepsilon(\cdot,t)$, $B^\varepsilon(\cdot,t)$ in $x$ and $t$. It will be useful in obtaining uniform convergence of approximation solutions later. The proof can be found in \cite{suenhoff12} pp. 56.

\begin{proposition}\label{Holder continuity proposition}
Given $\tau>0$ there is a constant $C=C(\tau)$ independent of $\varepsilon$ such that,
\begin{align}\label{Holder estimates in space time}
\langle u^\varepsilon\rangle^{\frac{1}{2},\frac{1}{4}}_{\R^3\times [\tau,\infty)},\ \langle B^\varepsilon\rangle^{\frac{1}{2},\frac{1}{4}}_{\R^3\times [\tau,\infty)}\le C(\tau)C_0^\theta.
\end{align}
\end{proposition}

The compactness of the approximate solutions $(\rho^\varepsilon,u^\varepsilon,B^\varepsilon)$ can then be summarised in the following lemma.

\begin{lemma}\label{Compactness lemma}
There is a sequence $\varepsilon_k\to 0$ and functions $u, B$ and $\rho$ such that as $k\to\infty$,
\begin{align}\label{strong convergence on u and B}
\mbox{ $u^{\varepsilon_k},B^{\varepsilon_k}\rightarrow u,B$ uniformly on compact sets in $\R^3\times (0,\infty)$};\end{align}
\begin{align}\label{weak convergence 1}
\nabla u^{\varepsilon_k}(\cdot,t),\nabla B^{\varepsilon_k}(\cdot,t),\nabla\omega^{\varepsilon_k}(\cdot,t)\rightharpoonup\nabla u(\cdot,t),\nabla B(\cdot,t),\nabla\omega(\cdot,t)
\end{align}
weakly in $L^2(\R^3)$
for all $t>0$; 
\begin{align}\label{weak convergence 2}
&\mbox{$\sigma^{\frac{1-s}{2}}\dot{u}^{\varepsilon_k},\sigma^{\frac{1-s}{2}}B_t^{\varepsilon_k},\sigma^{\frac{2-s}{2}}\nabla\dot{u}^{\varepsilon_k},\sigma^{\frac{2-s}{2}}\nabla B_t^{\varepsilon_k}\rightharpoonup\sigma^{\frac{1-s}{2}}\dot{u},\sigma^{\frac{1-s}{2}}B_t,\sigma^{\frac{2-s}{2}}\nabla\dot{u},\sigma^{\frac{2-s}{2}}\nabla B_t$}
\end{align}
weakly in $L^2(\R^3\times[0,\infty))$; and
\begin{align}\label{strong convergence on rho}
\rho^{\varepsilon_k}(\cdot,t)\to \rho(\cdot,t)
\end{align}
strongly in $L^2_{loc}(\R^3)$ for every $t\ge 0$. Here $\sigma(t)=\min\{1,t\}$ and $s\in(\frac{1}{2},1]$.
\end{lemma}

\begin{proof} The uniform convergence \eqref{strong convergence on u and B} follows from the bound \eqref{Holder estimates in space time} on $u$ and $B$ via a diagonal process, thus fixing the sequence $\{\varepsilon_k\}$. The statements in  \eqref{weak convergence 1} and \eqref{weak convergence 2} then follow for this same sequence from \eqref{Holder estimates in space time} and considerations based on the equality of weak-$L^2$ derivatives and distribution derivatives. The strong convergence \eqref{strong convergence on rho} of $\rho^{\varepsilon_k}$ for a further subsequence requires an argument given in Lions \cite{lions98} which was later extended by Feireisl \cite{feireisl04}. We omit the details here.
\end{proof}

\begin{proof}[\bf Proof of Theorem~\ref{Existence theorem}]
In view of Theorem~\ref{a priori bounds finite time theorem}, Proposition~\ref{Holder continuity proposition} and Lemma~\ref{Compactness lemma}, the limiting functions $(\rho,u,B)$ of Lemma~\ref{Compactness lemma} inherit the bounds in \eqref{nabla u B in L2}-\eqref{bound on weak solution}. It is also clear from the modes of convergence described in Lemma~\ref{Compactness lemma} that $(\rho,u,B)$ satisfies the weak forms \eqref{weak sol 1}-\eqref{weak sol 3} of \eqref{MHD1}-\eqref{MHD4} as well as the initial condition \eqref{initial data}. The continuity statement \eqref{rho in H-1}-\eqref{u B in L2} then follows easily from these weak forms together with the bounds in \eqref{pointwise bound on rho theorem}-\eqref{bound on weak solution}. 

It remains to prove the bound \eqref{bound on time integral on nabla u} on the piecewise $C^{\beta(t)}$ modulus of $\rho(\cdot)$ and time integral of $\|\nabla u(\cdot,t)\|_{L^\infty}$. We fix $T>0$ and $N>0$ such that 
$$\|\rho^\varepsilon_0-\tilde\rho\|_{C^\beta_{pw}}\le N.$$

We first decompose $u^\varepsilon$ as $u^\varepsilon=u^\varepsilon_{F}+u^\varepsilon_{P}$, where $u^\varepsilon_{F}$, $u^\varepsilon_{P}$ satisfy
\begin{align}\label{decomposition of u}
\left\{
 \begin{array}{lr}
(\mu+\lambda)\Delta (u^\varepsilon_{F})^{j}=F^\varepsilon_{x_j} +(\mu+\lambda)(\omega^\varepsilon)^{j,k}_{x_k}\\
(\mu+\lambda)\Delta (u^\varepsilon_P)^{j}=(P^\varepsilon-P(\tilde{\rho}))_{x_j}.\\
\end{array}
\right.
\end{align}
To bound $u^\varepsilon_{F}$ we apply the estimate \eqref{L infty bound general} on $F^\varepsilon$ to obtain that, for any $r>3$ and $t\in[0,T]$,
\begin{align*}
\int_{0}^{t}||\nabla u^\varepsilon_{F}(\cdot,\tau)||_{\infty}d\tau\le C(r)\int_{0}^{t}\left[||\nabla u^\varepsilon_{F}(\cdot,\tau)||_{L^r}+||D_{x}^2 u^\varepsilon_{F}(\cdot,\tau)||_{L^r}\right]d\tau.
\end{align*}
The right side of the above can be controlled by the time integrals on $\|\rho^\varepsilon\dot{u}^\varepsilon(\cdot,t)\|_{L^r}$ and $\|B^\varepsilon\cdot\nabla B^\varepsilon(\cdot,t)\|_{L^r}$, which can be bounded by the bounds \eqref{H1 bound} and \eqref{H2 bound}. Therefore we obtain
\begin{equation}\label{time integral on uF}
\int_{0}^{t}||\nabla u^\varepsilon_{F}(\cdot,\tau)||_{\infty}d\tau\le C(t)C_0^{\theta}.
\end{equation}

On the other hand, the way in controlling $u^\varepsilon_{P}$ is a bit more subtle than that of $u^\varepsilon_{F}$. For $t\in[0,T]$, the logical flow in obtaining the desired bound on $\int_{0}^{t}||\nabla u^\varepsilon_{P}(\cdot,\tau)||_{\infty}d\tau$ can be outlined as follows:
\begin{enumerate}
\item[Step 1.] $P^\varepsilon(\cdot,t)-\tilde P$ is pointwisely bounded independently of time.
\item[Step 2.] $u^\varepsilon_P(\cdot,t)$ is log-Lipschitz with bounded log-Lipschitz seminorm.
\item[Step 3.] The integral curve $x^\varepsilon(y,t)$ as defined by 
\begin{align*}
\left\{ \begin{array}
{lr} \dot{x}^\varepsilon(t,y)
=u^\varepsilon(x^\varepsilon(t,y),t)\\ x^\varepsilon(0,y)=y,
\end{array} \right.
\end{align*}
is H\"{o}lder-continuous in $y$.
\item[Step 4.] $\rho^\varepsilon(\cdot,t)$ is piecewise $C^{\beta(t)}$ with exponent $\beta(t)$ and modulus suitably bounded in finite time.
\item[Step 5.] The $C^{\beta(t)+1}(\R^3)$ norm of $u^\varepsilon_P(\cdot,t)$ is finite in finite time.
\end{enumerate}

Step~1 can be easily accomplished by making use of the pointwise bound \eqref{bound on initial density} on the density. Then by using Step~1 and the Poisson equation \eqref{decomposition of u}$_2$ on $u^\varepsilon_{P}$, we can apply the results from Bahouri-Chemin \cite{bahouriChemin94} to show that Step~2 holds for $u^\varepsilon_{P}$ with
\begin{equation}\label{bound on LL norm}
|u^\varepsilon_P(y,t)-u^\varepsilon_P(z,t)|\le Cm(|y-z|),
\end{equation}
where $C$ is a constant which depends only on $\bar \rho$ and $\tilde\rho$, and $m$ is given by
\begin{align*}
m(z)=\left\{
 \begin{array}{lr}
z(1-\log(z)),\; &0<z\le1;\\
z,\;&1<z<\infty.\\
\end{array}
\right.
\end{align*}

For Step~3, we apply \eqref{bound on LL norm} to obtain:
\begin{align*}
&\big|\frac{d}{dt}|x^\varepsilon(t,y)-x^\varepsilon(t,z)|^2\big|\\
&\le|x^\varepsilon(t,y)-x^\varepsilon(t,z)||u^\varepsilon_P(x^\varepsilon(t,y),t)-u^\varepsilon_P(x^\varepsilon(t,z),t)|\\
&\qquad+|x^\varepsilon(t,y)-x^\varepsilon(t,z)||u^\varepsilon_{F}(x^\varepsilon(t,y),t)-u^\varepsilon_{F}(x^\varepsilon(t,z),t)|\\
&\le\left[C+||\nabla u^\varepsilon_{F}(\cdot,t)||_{\infty}\right]|x^\varepsilon(t,y)-x^\varepsilon(t,z)|^{2}\left[1-\log|x^\varepsilon(t,y)-x^\varepsilon(t,z)|^2\right].
\end{align*}
Upon integrating the above differential inequality in time and utilizing the bound \eqref{time integral on uF} on $u_F$, there exists $\beta_1(t)$, $\beta_2(t)$, $\tilde C(t)>0$ such that
\begin{equation}\label{Holder norm of x(t) 1}
 |x^\varepsilon(t,y)-x^\varepsilon(t,z)|\le \tilde C(t)|y-z|^{\beta_1(t)}
 \end{equation}
 and 
 \begin{equation}\label{Holder norm of x(t) 2}
|y-z|\le \tilde C(t)|x^\varepsilon(t,y)-x^\varepsilon(t,z)|^{\beta_2(t)}.
\end{equation}

Next we proceed to Step~4. Let $y,z\in\R^3$ which are both on the same side of $\mathcal{C}(t)$. Then there exists $y_0,z_0\in\R^3$ such that
\begin{align*}
\left\{ \begin{array}
{lr} x^\varepsilon(t,y_0)
=y\\ x^\varepsilon(t,z_0)=z,
\end{array} \right.
\end{align*}
and $y_0,z_0$ are both on the same side of $\mathcal{C}(0)$. Integrating the mass equation along integral curves $x^\varepsilon(t,y_0)$ and $x^\varepsilon(t,z_0)$, subtracting and recalling the definition \eqref{definition of F} of $F$, we obtain that
\begin{align}\label{Holder norm estimate of rho}
|\log\rho^\varepsilon&(y,t)-\log\rho^\varepsilon(z,t)|\\
&\le|\log\rho^\varepsilon_0(y_0)-\log\rho^\varepsilon_0(z_0)|
+\int_0^t| P(\rho^\varepsilon_0(x^\varepsilon(\tau,y_0),\tau))-P(\rho^\varepsilon(x^\varepsilon(\tau,z_0),\tau)|d\tau\notag \\
&\qquad\qquad\qquad\qquad\qquad\qquad\quad+\int_0^t|F^\varepsilon(x^\varepsilon(\tau,y_0),\tau)-F^\varepsilon(x^\varepsilon(\tau,z_0),\tau)|d\tau.\notag
\end{align}
The first term on the right can be bounded in terms of $N$, and since $P$ is increasing, the second term is dissipative and can be dropped out. With the help of the estimate \eqref{bound on nabla F} on $F$ and the bound \eqref{Holder norm of x(t) 1}, the third term on the right is bounded as follows:
\begin{align*}
\int_0^t&|F^\varepsilon(x^\varepsilon(\tau,y_0),\tau)-F^\varepsilon(x^\varepsilon(\tau,z_0),\tau)|d\tau\\
&\le\int_0^t||\nabla F^\varepsilon(\cdot,\tau)||_{L^r}(\tilde C(t)|y_0-z_0|^{\beta_1(t)})^\alpha d\tau\\
&\le C\int_0^t(\|\rho^\varepsilon\dot{u}^\varepsilon\|_{L^r}+\|B^\varepsilon\cdot\nabla B^\varepsilon\|_{L^r})(\tilde C(t)|y_0-z_0|^{\beta_1(t)})^\alpha d\tau,
\end{align*}
where $r>3$ and $\alpha=1-\frac{3}{r}$. Using the bounds in \eqref{bound on weak solution finite time}-\eqref{pointwise bound on rho finite time}, the term involving $\dot{u}^\varepsilon$ can be bounded by
\begin{align*}
&C\int_0^t(\|\dot{u}^\varepsilon(\cdot,\tau)\|^{\frac{1-\delta}{2}}_{L^2}\|\nabla\dot{u}^\varepsilon(\cdot,\tau)\|^{\frac{1-\delta}{2}}_{L^2})d\tau\\
&\le C\int_0^t \tau^\gamma\Big(\tau^{1-s}\intox|\dot{u}^\varepsilon|^2dx\Big)^\frac{1-\delta}{4}\Big(\tau^{2-s}\intox|\nabla\dot{u}^\varepsilon|^2dx\Big)^\frac{1+\delta}{4}d\tau\\
&\le C\Big(\int_0^t\tau^{2\gamma}d\tau\Big)^\frac{1}{2}C^{\bar{\theta}},
\end{align*}
for some $\delta>0$ and $4\gamma=(s-1)(1-\delta)-(2-s)$. Since $2\gamma>-1$ if $s>\frac{1}{2}$, the above time integral is finite and hence we obtain
\begin{equation*}
\int_0^t\|\rho^\varepsilon\dot{u}^\varepsilon\|_{L^r}(\tilde C(t)|y_0-z_0|^{\beta_1(t)})^\alpha d\tau\le C(t)|y_0-z_0|^{\alpha\beta_1(t)},
\end{equation*}
for some $C(t)>0$. The estimate for the term involving $B^\varepsilon\cdot\nabla B^\varepsilon$ is just similar, so by applying the bound \eqref{Holder norm of x(t) 2}, we obtain from \eqref{Holder norm estimate of rho} that, 
\begin{align}\label{bound on Holder norm of rho}
|\log\rho^\varepsilon(y,t)-\log\rho^\varepsilon(z,t)|&\le \tilde C(t)N|y-z|^{\beta_0\beta_2(t)}+\tilde C(t)C(t)|y-z|^{\alpha\beta_1(t)\beta_2(t)}\notag\\
&\le C(N,t,C_0)|y-z|^{\beta(t)}
\end{align}
for some $C(N,t,C_0)>0$ and $\beta(t)>0$. The above shows that $\rho^\varepsilon(\cdot,t)$ is piecewise $C^{\beta(t)}$ with bounded modulus and Step~4 is completed.

Finally, with the improved regularity on $\rho^\varepsilon(\cdot,t)$ from Step~4, we can now make use of \eqref{decomposition of u}$_2$ again and apply properties of Newtonian potentials to conclude that the $C^{\beta(t)+1}(\R^3)$ norm of $u_P$ remains finite in finite time, which finishes Step~5 as described above.

We therefore obtain the bound on $u_p$:
\begin{equation}\label{time integral on uP}
\int_{0}^{t}||\nabla u^\varepsilon_{P}(\cdot,\tau)||_{\infty}d\tau\le C(N,t,C_0).
\end{equation}
for some $C(t,N,C_0)>0$. Combining \eqref{time integral on uF}, \eqref{bound on Holder norm of rho} and \eqref{time integral on uP}, there exists a constant $C(N,T,C_0)>0$ such that
\begin{equation}\label{bound on time integral on nabla u approximation}
\sup_{0\le \tau\le T}\|\rho^\varepsilon(\cdot,t)-\tilde\rho\|_{C^{\beta(t)}_{pw}}+\int_0^T\|\nabla u^\varepsilon(\cdot,\tau)\|_{L^\infty}d\tau\le C(N,T,C_0)
\end{equation}
for all $\varepsilon>0$. Notice that $\beta$ and $C(N,T,C_0)$ are all independent of $\varepsilon$. Hence by taking $\varepsilon\to0$ (or some subsequence $\varepsilon_k\to0$), the limit $(\rho,u,B)$ satisfies \eqref{bound on time integral on nabla u} and we finish the proof of Theorem~\ref{Existence theorem}.
\end{proof}

\section{Uniqueness of weak solution and proof of Theorem~\ref{Uniqueness theorem}}\label{proof of uniqueness section}

In this section, we address the uniqueness of weak solutions to \eqref{MHD1}-\eqref{MHD4} and prove Theorem~\ref{Uniqueness theorem}. To begin with, we fix constants $L, \beta_0, N, \rho_1,\rho_2,\tilde\rho>0$, $q>6$, $s\in(\frac{1}{2},1]$ and let $(\rho_0,u_0,B_0)$ and $(\bar{\rho}_0,\bar{u}_0,\bar{B}_0)$ be initial data satisfying \eqref{bound on initial density}-\eqref{definition of C0} and \eqref{piecewise holder for initial rho}. By Theorem~\ref{Existence theorem}, there exists weak solutions $(\rho,u,B)$ and $(\bar{\rho},\bar{u},\bar{B})$ to \eqref{MHD1}-\eqref{MHD4} with initial data $(\rho_0,u_0,B_0)$ and $(\bar{\rho}_0,\bar{u}_0,\bar{B}_0)$ respectively, both satisfying \eqref{rho in H-1}-\eqref{bound on weak solution}. In particular, for each $T>0$, we have
\begin{equation}\label{bound on A(T) uniqueness}
\bar{\mathcal{A}}(T)+\mathcal{A}(T)\le C,
\end{equation}
\begin{equation}\label{bound on rho uniqueness}
C^{-1}\le\|\rho\|_{L^\infty(\R^3\times[0,T])}\le C,\qquad C^{-1}\le\|\bar{\rho}\|_{L^\infty(\R^3\times[0,T])}\le C,
\end{equation}
\begin{equation}\label{bound on time integral on nabla u and bar u}
\int_0^T(\|\nabla u(\cdot,\tau)\|_{L^\infty}+\|\nabla \bar{u}(\cdot,\tau)\|_{L^\infty})d\tau\le C,
\end{equation}
for some constant $C$ which may depend on $T$, $L$, $\beta_0$, $N$, $\rho_1$, $\rho_2$, $\tilde\rho$, $s$ and $C_0$, and $\bar{\mathcal{A}}(T)$ is defined by \eqref{def of A(t)} with $(\rho,u,B)$ being replaced by $(\bar{\rho},\bar{u},\bar{B})$. We will then obtain the uniqueness and continuous dependence on initial data of weak solutions by proving the assertion \eqref{estimate on difference of solutions} for $(\rho,u,B)$ and $(\bar{\rho},\bar{u},\bar{B})$. 

We follow the idea and analysis given in Hoff \cite{hoff06} for compressible Navier-Stokes equations, which suggested that solutions with minimal regularity are best compared in a Lagrangian framework. We therefore state the following proposition about integral curves given in \eqref{integral curve of u}. More precisely, for $T>0$, the bound \eqref{bound on time integral on nabla u and bar u} guarantees the existence and uniqueness of the mapping $X(y,t,t')\in C(\R^3\times[0,T]^2)$ satisfying
\begin{align}\label{integral curve X}
\left\{ \begin{array}
{lr} \dis\frac{\partial X}{\partial t}(y,t,t')
=u(X(y,t,t'),t)\\ X(y,t',t')=y
\end{array} \right.
\end{align}
where $(\rho,u,B)$ is a weak solution to \eqref{MHD1}-\eqref{MHD4} on $\R^3\times[0,T]$ satisfying \eqref{rho in H-1}-\eqref{bound on weak solution} and \eqref{bound on time integral on nabla u}. Moreover, using the bound \eqref{bound on time integral on nabla u and bar u}, the mapping $X(\cdot,t,t')$ is Lipschitz on $\R^3$ for $(t,t')\in[0,T]^2$.

\begin{proposition}\label{prop on X}
Let $T>0$ and $u$ satisfy \eqref{bound on time integral on nabla u and bar u}. Then there is a unique function $X\in C(\R^3\times[0,T]^2$) satisfying \eqref{integral curve X}. In particular, $X(\cdot,t,t')$ is Lipschitz on $\R^3$ for $(t,t')\in[0,T]^2$, and there is a constant $C$ such that
\begin{equation*}
\Big\|\frac{\partial X}{\partial y}(\cdot,t,t')\Big\|_{L^\infty}\le C,\qquad (t,t')\in[0,T]^2.
\end{equation*}
\end{proposition} 

\begin{proof}
Refer to the proof of Lemma~2.1 in Hoff \cite{hoff06}.
\end{proof}

With respect to velocities $u$ and $\bar{u}$, for $y\in\R^3$, we let $X$, $\bar{X}$ be two integral curves given by
\begin{align*}
\left\{ \begin{array}
{lr} \dis\frac{\partial X}{\partial t}(y,t,t')
=u(X(y,t,t'),t)\\ X(y,t',t')=y
\end{array} \right.
\end{align*}
and
\begin{align*}
\left\{ \begin{array}
{lr} \dis\frac{\partial \bar{X}}{\partial t}(y,t,t')
=\bar{u}(\bar{X}(y,t,t'),t)\\ \bar{X}(y,t',t')=y.
\end{array} \right.
\end{align*}
We then define $S(x,t)$, $S^{-1}(x,t)$ by
\begin{equation}\label{def of S}
S(x,t)=\bar{X}(X(x,0,t),t,0),
\end{equation}
and 
\begin{equation}\label{def of S-1}
S^{-1}(x,t)=X(\bar{X}(x,0,t),t,0).
\end{equation}
The following proposition provides some properties of $S$ and $S^{-1}$ which will become useful later. A proof can be found in Hoff \cite{hoff06}.
\begin{proposition}\label{prop on S}
Let $S$ and $S^{-1}$ be as given in \eqref{def of S}-\eqref{def of S-1}. Then we have
\begin{itemize}
\item $S^{\pm1}$ is continuous on $R^3\times[0,T]$ and Lipschitz continuous on $R^3\times[\tau,T]$ for all $\tau>0$, and there is a constant C such that $$\|\nabla S^{\pm1}(\cdot,t)\|_{L^\infty}\le C,\qquad t\in[0,T];$$
\item $(S_t+\nabla S u)(x,t)=\bar{u}(S(x,t),t)$ a.e. in $\R^3\times(0,T);$
\item $\bar{\rho}(S(x,t),t)\rho_0(X(x,0,t))\det\nabla S(x,t)=\rho(x,t)\bar{\rho}_0(X(x,0,t))$ a.e. in $\R^3\times(0,T).$
\end{itemize}
\end{proposition}
We are now ready to give the proof of Theorem~\ref{Uniqueness theorem}. First, we let $\psi:\R^3\times[0,T]\rightarrow\R^3$ be a test function satisfying
\begin{align}\label{weak form for rho}
&-\intox \rho_0(x)u_0(x)\psi(x,0)dx\notag\\
&=\intoxt \Big[\rho u\cdot(\psi_t+\nabla\psi u)+(P(\rho)-\tilde P)\divv(\psi)-\mu\nabla u^j\cdot\nabla\psi^j\notag\\
&\qquad\qquad\qquad\qquad-\lambda\divv(u)\divv(\psi)+(\frac{1}{2}|B|^2)\divv(\psi)-B^jB\cdot\nabla\psi^j\Big]dxd\tau.
\end{align}
Define $\bar\psi=\psi\circ S^{-1}$. Then we have
\begin{align}\label{weak form for bar rho}
&-\intox \bar{\rho}_0(x)\bar{u}_0(x)\bar{\psi}(x,0)dx\notag\\
&=\intoxt \Big[\bar{\rho}\bar{u}\cdot(\bar{\psi}_t+\nabla\bar{\psi}\bar{u})+(P(\bar{\rho})-\tilde P)\divv(\bar{\psi})-\mu\nabla\bar{u}^j\cdot\nabla\bar{\psi}^j\notag\\
&\qquad\qquad\qquad\qquad-\lambda\divv(\bar{u})\divv(\bar{\psi})+(\frac{1}{2}|\bar{B}|^2)\divv(\bar{\psi})-\bar{B}^j\bar{B}\cdot\nabla\bar{\psi}^j\Big]dxd\tau.
\end{align}
Notice that
\begin{align*}
\intox \bar{\rho}\bar{u}\cdot(\bar{\psi}_t+\nabla\bar{\psi}\bar{u})dx&=\intox \bar{\rho}(S)\bar{u}(S)\cdot(\bar{\psi}_t(S)+\nabla\bar{\psi}\bar{u}(S))|\det(\nabla S)|dx\\
&=\intox A_0\rho\bar{u}(S)(\psi_t+\nabla\psi u)dx,
\end{align*}
where we used the fact that $A_0\rho=(\bar{\rho}\circ S)|\det(\nabla S)|$ from Proposition~\ref{prop on S}. Hence by taking the difference between \eqref{weak form for rho} and \eqref{weak form for bar rho}, and using the effective viscous flux $\bar{F}$ as defined in \eqref{definition of F} (replacing $F$ by $\bar{F}$, $u$ by $\bar{u}$, etc.), for all $\psi$ and $\bar{\psi}$, we have
\begin{align}\label{estimate on weak form for u and rho}
&\intox (\bar{\rho}\bar{u}_0-\rho_0u_0)\cdot\psi(x,0)dx\notag\\
&=\intoxt \Big[\rho(u-\bar{u}\circ S)(\psi_t+\nabla\psi u)+(1-A_0)\rho(\bar{u}\circ S)(\psi_t+\nabla\psi u)\Big]dxd\tau\notag\\
&\qquad+\intoxt \Big[(\tilde P-P(\rho))\divv(\bar{\psi})+\mu\nabla\bar{u}^j\cdot\nabla\bar{\psi}^j+\lambda\divv(\bar{u})\divv(\bar{\psi})\Big]dxd\tau\notag\\
&\qquad+\intoxt \Big[(1-A_0)\rho(\bar{u}\circ S)(\psi_t+\nabla\psi u)+(P(\rho)-P(\bar{\rho}))\divv(\psi)\Big]dxd\tau\notag\\
&\qquad+\intoxt \Big(\frac{1}{2}|B|^2\divv(\psi)-\frac{1}{2}|\bar{B}|^2\divv(\bar{\psi})\Big)-\Big(B^jB\cdot\nabla\psi^j-\bar{B}^j\bar{B}\cdot\nabla\bar{\psi}^j\Big)dxd\tau\notag\\
&\qquad+\intoxt \Big[\nabla\bar{F}\cdot(\psi-\psi\circ S^{-1})+\mu\bar{\omega}^{j,k}_{x_k}(\psi^j-\psi^j\circ S^{-1})\Big]dxd\tau\notag\\
&\qquad+\intoxt (\bar{u}\circ S-\bar{u})(\mu\Delta\psi+\lambda\nabla\divv(\psi))dxd\tau.
\end{align}
Similarly, with respect to the magnetic fields $B$ and $\bar{B}$, we let $\varphi:\R^3\times[0,T]\to\R^3$ be test function satisfying
\begin{align}\label{estimate on weak form for B}
&-\intox (\bar{B}_0-B_0(x))\cdot\varphi(x,0)dx\notag\\
&=\intoxt(B-\bar{B})\cdot(\varphi_t+u\cdot\nabla\varphi+\nu\Delta\varphi)dxd\tau+\intoxt\nabla\varphi^j(\bar{B}-B)\bar{B}^jdxd\tau\notag\\
&\qquad+\intoxt\nabla\varphi^j(\bar{u}^j-u^j)Bdxd\tau+\intoxt\nabla\varphi^j(u-\bar{u})\bar{B}^jdxd\tau.
\end{align}

Next we extend $\rho$, $u$ and $B$ to be constant in $t$ outside $[0,T]$ and let $\rho^\varepsilon$, $u^\varepsilon$ and $B^\varepsilon$ be the corresponding smooth approximation obtained by mollifying in both $x$ and $t$. Then we define $\psi^\varepsilon,\varphi^\varepsilon:\R^3\times[0,T]\to\R^3$ to be the solutions satisfying
\begin{align*}
\left\{ \begin{array}
{lr} \rho^\varepsilon(\psi^\varepsilon_t+u^\varepsilon\cdot\nabla\psi^\varepsilon)+\mu\Delta\psi^\varepsilon+\lambda\nabla\divv(\psi^\varepsilon)=G\\ 
\psi^\varepsilon(\cdot,T)=0,
\end{array} \right.
\end{align*}
and
\begin{align*}
\left\{ \begin{array}
{lr} \varphi^\varepsilon_t+u^\varepsilon\cdot\nabla\varphi^\varepsilon+\nu\Delta\varphi^\varepsilon=H\\ 
\varphi^\varepsilon(\cdot,T)=0,
\end{array} \right.
\end{align*}
for given functions $G,H\in H^\infty(\R^3\times[0,T])$. By simple estimates, $\psi^\varepsilon$ and $\varphi^\varepsilon$ satisfy the following bounds in terms of $G$ and $H$:
\begin{align}\label{bound on psi}
&\sup_{0\le \tau\le T}\intox[|\psi^\varepsilon(x,t)|^2+|\nabla\psi^\varepsilon(x,t)|^2]dx+\intoxt[|\psi^\varepsilon_t+\nabla\psi^\varepsilon u^\varepsilon|^2+|D^2_x\psi^\varepsilon|^2]dxd\tau\notag\\
&\le C\intoxt|G|^2dxd\tau
\end{align}
and
\begin{align}\label{bound on varphi}
&\sup_{0\le \tau\le T}\intox[|\varphi^\varepsilon(x,t)|^2+|\nabla\varphi^\varepsilon(x,t)|^2]dx+\intoxt[|\varphi^\varepsilon_t+\nabla\varphi^\varepsilon u^\varepsilon|^2+|D^2_x\varphi^\varepsilon|^2]dxd\tau\notag\\
&\le C\intoxt|H|^2dxd\tau.
\end{align}

We now take $\psi=\psi^\varepsilon$ in \eqref{estimate on weak form for u and rho} and $\varphi=\varphi^\varepsilon$ in \eqref{estimate on weak form for B} respectively to obtain
\begin{align}\label{estimate on weak form for u and rho 2}
\intox (\bar{\rho}\bar{u}_0-\rho_0u_0)\cdot\psi^\varepsilon(x,0)dx=\intoxt z\cdot Gdxd\tau+\sum_{i=1}^7\mathcal{R}_i,
\end{align}
\begin{align}\label{estimate on weak form for B 2}
-\intox (\bar{B}_0-B_0)\cdot\varphi^\varepsilon(x,0)dx=\intoxt(B-\bar{B})\cdot Hdxd\tau+\mathcal{R}_8,
\end{align}
where $z=u-\bar{u}\circ S$ and $\mathcal{R}_1,\dots,\mathcal{R}_8$ are given by:
\begin{align*}\
&\mathcal{R}_1=\intoxt\Big[\nabla\bar{F}\cdot(\psi^\varepsilon-\psi^\varepsilon\circ S^{-1})+\mu\bar{\omega}^{j,k}_{x_k}(\psi^\varepsilon-\psi^\varepsilon\circ S^{-1})\Big]dxd\tau,\\
&\mathcal{R}_2=\intoxt z\cdot \Big[(\rho-\rho^\varepsilon)\psi^\varepsilon_t+\nabla\psi^\varepsilon(\rho u-\rho^\varepsilon u^\varepsilon)\Big]dxd\tau,
\end{align*}
\begin{align*}
&\mathcal{R}_3=\intoxt(\bar{u}\circ S-\bar{u})\cdot(\mu\Delta\psi^\varepsilon+\lambda\divv(\psi^\varepsilon))dxd\tau,\\
&\mathcal{R}_4=\intoxt(1-A_0)\rho(\bar{u}\circ S)\cdot(\psi^\varepsilon_t+\nabla\psi^\varepsilon u)dxd\tau,\\
&\mathcal{R}_5=\intoxt(P(\rho)-P(\bar\rho))\divv(\psi^\varepsilon)dxd\tau,
\end{align*}
\begin{align*}
&\mathcal{R}_6=\intoxt\bar{B}^j_{x_k}\bar{B}^k((\psi^\varepsilon)^j-(\psi^\varepsilon)^j\circ S^{-1})dxd\tau\\
&\qquad\qquad\qquad-\intoxt\frac{1}{2}\nabla(|\bar{B}|^2)\cdot(\psi^\varepsilon-\psi^\varepsilon\circ S^{-1})dxd\tau,\\
&\mathcal{R}_7=\intoxt\frac{1}{2}(|B|^2-|\bar{B}|^2)\divv(\psi^\varepsilon)dxd\tau\\
&\qquad\qquad\qquad-\intoxt(B^jB-\bar{B}^j\bar{B})\cdot\psi^\varepsilon dxd\tau,
\end{align*}
and
\begin{align*}
\mathcal{R}_8&=\intoxt\nabla(\varphi^\varepsilon)^j(\bar{B}-B)\bar{B}^jdxd\tau+\intoxt\nabla(\varphi^\varepsilon)^j(\bar{u}^j-u^j)Bdxd\tau\\
&\qquad+\intoxt\nabla(\varphi^\varepsilon)^j(u-\bar{u})\bar{B}^jdxd\tau.
\end{align*}

Our main goal is to estimate the terms $\mathcal{R}_1,\dots,\mathcal{R}_{8}$ and the terms on the left sides of \eqref{estimate on weak form for u and rho 2}-\eqref{estimate on weak form for B 2} and then take the limit as $\varepsilon\to0$. Most of the analysis are reminiscent of those given in Hoff \cite{hoff06} except the terms $\mathcal{R}_1$, $\mathcal{R}_3$, $\mathcal{R}_6$, $\mathcal{R}_7$ and $\mathcal{R}_8$. 

Following the steps given in \cite{hoff06}, using Proposition~\ref{prop on S} and applying the bound \eqref{bound on psi} on $\psi^\varepsilon$, we are ready to obtain:
\begin{align}\label{estimate of LHS}
&\Big|\intox (\bar{\rho}\bar{u}_0-\rho_0u_0)\cdot\psi^\varepsilon(x,0)dx\Big|+\Big|\intox (\bar{B}_0-B_0)\cdot\varphi^\varepsilon(x,0)dx\Big|\notag\\
&\le\|\rho_0u_0-\bar{\rho}_0\bar{u}_0\|_{L^2}\Big(\intoxt|G|^2dxd\tau\Big)^\frac{1}{2}+\|B_0-\bar{B}_0\|_{L^2}\Big(\intoxt|H|^2dxd\tau\Big)^\frac{1}{2},
\end{align}
\begin{equation}\label{estimate on R4}
|\mathcal{R}_4|\le CT^\frac{1}{2}\Big[\|\rho_0-\bar{\rho}_0\|_{L^2}+\Big(\intoxt|z|^2dxd\tau\Big)^\frac{1}{2}\Big]\Big(\intoxt|G|^2dxd\tau\Big)^\frac{1}{2},
\end{equation}
\begin{equation}\label{estimate on R2}
\lim_{\varepsilon\to0}\mathcal{R}_2=0,
\end{equation}
and under the assumption \eqref{isothermal} on pressure $P$, the following estimate holds
\begin{align}\label{estimate on R5}
|\mathcal{R}_5|&\le\intoxt K|(\rho-\bar\rho)\divv(\psi^\varepsilon)|dxd\tau\notag\\
&\le CT^\frac{1}{2}\Big[\|\rho_0-\bar{\rho}_0\|_{L^2}+\Big(\intoxt|z|^2dxd\tau\Big)^\frac{1}{2}\Big]\Big(\intoxt|G|^2dxd\tau\Big)^\frac{1}{2}.
\end{align}

We now give the estimates $\mathcal{R}_1$, $\mathcal{R}_3$, $\mathcal{R}_6$, $\mathcal{R}_7$ and $\mathcal{R}_8$ as follows. To estimate $\mathcal{R}_1$, modulo the vorticity $\omega$, we obtain that
\begin{align*}
|\mathcal{R}_1|&\le C\Big(\intoxt|z|^2dxd\tau\Big)^\frac{1}{2}\int_0^T \tau^\frac{1}{2}\|\nabla\bar{F}(\cdot,\tau)\|_{L^4}\|\nabla\psi^\varepsilon(\cdot,t)\|_{L^4}d\tau\\
&\le C\Big(\intoxt|z|^2dxd\tau\Big)^\frac{1}{2}\Big(\intoxt|G|^2dxd\tau\Big)^\frac{1}{2}\Big(\int_0^T \tau^\frac{4}{5}\|\nabla\bar{F}(\cdot,\tau)\|^\frac{8}{5}_{L^4}d\tau\Big)^\frac{5}{8}\\
&\qquad\times\Big(\intoxt|D^2_x\psi^\varepsilon|^2dxd\tau\Big)^\frac{3}{8}.
\end{align*}
To bound the term involving $\bar{F}$ as above, we use \eqref{bound on A(T) uniqueness}-\eqref{bound on rho uniqueness} to obtain
\begin{equation}\label{bound on L4 nabla F}
\int_0^T \tau^\frac{4}{5}\|\nabla\bar{F}(\cdot,\tau)\|^\frac{8}{5}_{L^4}d\tau\le\int_0^T \tau^\frac{4}{5}\Big(\intox|\dot{\bar{u}}|^4dx+\intox|\nabla\bar{B}|^4|\bar{B}|^4dx\Big)^\frac{3}{5}d\tau.
\end{equation}
The first integral on the right side of \eqref{bound on L4 nabla F} is bounded by
\begin{align*}
&\int_0^T \tau^\frac{4}{5}\Big(\intox|\dot{\bar{u}}|^4dx\Big)^\frac{3}{5}d\tau\\
&\le C\int_0^T \tau^\frac{4}{5}\Big(\intox|\dot{\bar{u}}|^2dx\Big)^\frac{1}{5}\Big(\intox|\nabla\dot{\bar{u}}|^2dx\Big)^\frac{3}{5}d\tau\\
&\le C\Big(\int_0^T \tau^{4s-3}d\tau\Big)^\frac{1}{5}\Big(\int_0^T \tau^{1-s}\intox|\dot{\bar{u}}|^2dxd\tau\Big)^\frac{1}{5}\Big(\int_0^T \tau^{1-s}\intox|\dot{\bar{u}}|^2dxd\tau\Big)^\frac{1}{5}\le CT^\frac{4s-2}{5},
\end{align*}
where the last inequality holds by the bounds \eqref{bound on weak solution} and \eqref{bound on A(T) uniqueness}-\eqref{bound on rho uniqueness}, and the assumption that $s>\frac{1}{2}$. On the other hand, to bound the term involving $B\cdot\nabla B$ in \eqref{bound on L4 nabla F}, we have
\begin{align*}
&\int_0^T \tau^\frac{4}{5}\Big(\intox|\nabla B|^4|B|^4dx\Big)^\frac{2}{5}d\tau\\
&\le C\Big(\int_0^T \tau^{4s-3}d\tau\Big)^\frac{1}{5}\Big(\int_0^T\tau^{1-s}\intox|\nabla B|^2|B|^2dxd\tau\Big)^\frac{1}{5}\\
&\qquad\times\Big(\int_0^T \tau^{2-s}\intox|\nabla B|^4dxd\tau+\int_0^T \tau^{2-s}\intox|\Delta B|^2|B|^2dxd\tau\Big)^\frac{3}{5}.
\end{align*}
Following the proof of the bound \eqref{bound on A123} in Lemma~\ref{H2 estimates}, we use \eqref{bound on A(T) uniqueness}-\eqref{bound on rho uniqueness} to obtain
\begin{align*}
\int_0^T\tau^{1-s}\intox|\nabla B|^2|B|^2dxd\tau+\int_0^T \tau^{2-s}\intox|\nabla B|^4dxd\tau\le C.
\end{align*}
Using the magnetic field equation \eqref{MHD3}, we also have
\begin{align*}
&\int_0^T \tau^{2-s}\intox|\Delta B|^2|B|^2dxd\tau\\
&\le C\Big(\intoxt \tau^{2-s}|B_t|^2|B|^2dxd\tau+\intoxt \tau^{2-s}(|\nabla u|^2|B|^2dxd\tau+|\nabla B|^2|u|^2)|B|^2dxd\tau\Big)\\
&\le C+C\intoxt \tau^{2-s}|B_t|^2|B|^2dxd\tau,
\end{align*}
and also
\begin{align*}
&\intoxt \tau^{2-s}|B_t|^2|B|^2dxd\tau\\
&\le\int_0^T \tau^{2-s}\Big(\intox|B_t|^3dx\Big)^\frac{2}{3}\Big(\intox|B|^6dx\Big)^\frac{1}{3}d\tau\\
&\le C\Big(\sup_{0\le \tau\le T}\intox|B|^6dx\Big)^\frac{1}{3}\Big(\intoxt \tau^{2-s}|\nabla B_t|^2dxd\tau\Big)^\frac{1}{2}\Big(\intoxt \tau^{1-s}|B_t|^2dxd\tau\Big)^\frac{1}{2}\\
&\le C.
\end{align*}
Combining the above estimates, we conclude from \eqref{bound on L4 nabla F} that 
\begin{equation*}
\int_0^T \tau^\frac{4}{5}\|\nabla\bar{F}(\cdot,\tau)\|^\frac{8}{5}_{L^4}d\tau\le CT^\frac{2s-1}{4},
\end{equation*}
and we obtain the estimate on $\mathcal{R}_1$
\begin{equation}\label{estimate on R1}
|\mathcal{R}_1|\le CT^\frac{2s-1}{4}\Big(\intoxt|z|^2dxd\tau\Big)^\frac{1}{2}\Big(\intoxt|G|^2dxd\tau\Big)^\frac{1}{2}.
\end{equation}
In particular, for $[t_1,t_2]\subseteq[0,T]$, if we define
\begin{equation*}
\mathcal{R}_1(t_1,t_2)=\int_{t_1}^{t_2}\intox\Big[\nabla\bar{F}\cdot(\psi^\varepsilon-\psi^\varepsilon\circ S^{-1})+\mu\bar{\omega}^{j,k}_{x_k}(\psi^\varepsilon-\psi^\varepsilon\circ S^{-1})\Big]dxd\tau,
\end{equation*}
then we also have
\begin{equation*}
|\mathcal{R}_1(t_1,t_2)|\le C|t_2-t_1|^\frac{2s-1}{4}\Big(\intoxts|z|^2dxd\tau\Big)^\frac{1}{2}\Big(\intoxts|G|^2dxd\tau\Big)^\frac{1}{2}
\end{equation*}
with $C$ being independent of $t_1$, $t_2$ and $G$. The term $\mathcal{R}_3$ can be bounded in a similar way as $\mathcal{R}_1$.

To estimate $\mathcal{R}_6$, in view of the definition of $\mathcal{R}_6$, we first consider the term $\intoxt\bar{B}^j_{x_k}\bar{B}^k\Big((\psi^\varepsilon)^j-(\psi^\varepsilon)^j\circ S^{-1}\Big)$ which can be bounded as follows.
\begin{align*}
&\Big|\intoxt\bar{B}^j_{x_k}\bar{B}^k\Big((\psi^\varepsilon)^j-(\psi^\varepsilon)^j\circ S^{-1}\Big)dxd\tau\Big|\\
&\le C\intoxt|\nabla\bar{B}||\bar{B}||\psi^\varepsilon(x,\tau)-\psi^\varepsilon(S^{-1}(x,\tau),\tau)|dxd\tau\\
&\le C\Big(\intoxt|z|^2dxd\tau\Big)^\frac{1}{2}\int_0^T t^\frac{1}{2}\|\nabla\bar{B}(\cdot,\tau)\|_{L^4}\|\nabla\psi^\varepsilon(\cdot,\tau)\|_{L^4}dt\\
&\le C\Big(\intoxt|z|^2dxd\tau\Big)^\frac{1}{2}\Big(\intoxt|G|^2dxd\tau\Big)^\frac{1}{2}\Big(\int_0^T \tau^\frac{4}{5}\|\nabla\bar{B}(\cdot,\tau)\|_{L^4}^\frac{8}{5}d\tau\Big)^\frac{5}{8}.
\end{align*}
The term $\int_0^T \tau^\frac{4}{5}\|\nabla\bar{B}(\cdot,t)\|_{L^4}^\frac{8}{5}d\tau$ can be bounded in a similar way as $\int_0^T \tau^\frac{4}{5}\|\nabla\bar{F}(\cdot,t)\|^\frac{8}{5}_{L^4}d\tau$, hence we obtain
\begin{align*}
&\Big|\intoxt\bar{B}^j_{x_k}\bar{B}^k\Big((\psi^\varepsilon)^j-(\psi^\varepsilon)^j\circ S^{-1}\Big)dxd\tau\Big|\\
&\le CT^\frac{2s-1}{4}\Big(\intoxt|z|^2dxd\tau\Big)^\frac{1}{2}\Big(\intoxt|G|^2dxd\tau\Big)^\frac{1}{2}.
\end{align*}
On the other hand, the term $\intoxt\frac{1}{2}\nabla(|\bar{B}|^2)\cdot(\psi-\psi\circ S^{-1})dxd\tau$ in the definition of $\mathcal{R}_6$ can be treated similarly, and hence we conclude
\begin{equation}\label{estimate on R6}
|\mathcal{R}_6|\le CT^\frac{2s-1}{4}\Big(\intoxt|z|^2dxd\tau\Big)^\frac{1}{2}\Big(\intoxt|G|^2dxd\tau\Big)^\frac{1}{2}
\end{equation}
and in particular
\begin{equation*}
|\mathcal{R}_6(t_1,t_2)|\le C|t_2-t_1|^\frac{2s-1}{4}\Big(\intoxts|z|^2dxd\tau\Big)^\frac{1}{2}\Big(\intoxts|G|^2dxd\tau\Big)^\frac{1}{2},
\end{equation*}
where $[t_1,t_2]\subseteq[0,T]$ and $\mathcal{R}_6(t_1,t_2)$ is given by
\begin{equation*}
\mathcal{R}_6(t_1,t_2)=\int_{t_1}^{t_2}\intox\Big[\bar{B}^j_{x_k}\bar{B}^k((\psi^\varepsilon)^j-(\psi^\varepsilon)^j\circ S^{-1})-\frac{1}{2}\nabla(|\bar{B}|^2)\cdot(\psi^\varepsilon-\psi^\varepsilon\circ S^{-1})\Big]dxd\tau.
\end{equation*}

To estimate $\mathcal{R}_7$, we can readily obtain the bound as follows:
\begin{align*}
|\mathcal{R}_7|&\le C\Big(\intoxt|B-\bar{B}|^2dxd\tau\Big)^\frac{1}{2}\Big(\intoxt(|B|^6+|\bar{B}|^6)dxd\tau\Big)^\frac{1}{6}\Big(\intoxt|\nabla\psi^\varepsilon|^3dxd\tau\Big)^\frac{1}{3}
\end{align*}
which gives
\begin{equation}\label{estimate on R7}
|\mathcal{R}_7|\le CT^\frac{1}{3}\Big(\intoxt|G|^2dxd\tau\Big)^\frac{1}{2}.
\end{equation}

Similarly, for the term $\mathcal{R}_8$, using the bound \eqref{bound on varphi} we have the estimate
\begin{align}\label{estimate on R8}
|\mathcal{R}_8|&\le C\Big(\intoxt(|u-\bar{u}|^2+|B-\bar{B}|^2)dxd\tau\Big)^\frac{1}{2}\notag\\
&\qquad\times\Big(\intoxt(|B|^6+|\bar{B}|^6)dxd\tau\Big)^\frac{1}{6}\Big(\intoxt|\nabla\varphi^\varepsilon|^3dxd\tau\Big)^\frac{1}{3}\notag\\
&\le CT^\frac{1}{3}\Big(\intoxt|H|^2dxd\tau\Big)^\frac{1}{2}.
\end{align}
Summarizing the estimates \eqref{estimate of LHS}, \eqref{estimate on R4}, \eqref{estimate on R2}, \eqref{estimate on R5}, \eqref{estimate on R1}, \eqref{estimate on R6}, \eqref{estimate on R7} and \eqref{estimate on R8}, we arrive at
\begin{align}\label{estimate on z in T}
&\Big|\intoxt z\cdot Gdxd\tau\Big|\notag\\
&\le C\Big[M_0\Big(\intoxt|G|^2dxd\tau\Big)^\frac{1}{2}+|\mathcal{R}_1(0,T)|+|\mathcal{R}_6(0,T)|\Big],
\end{align}
and
\begin{align*}
\Big|\intoxt (B-\Bar{B})\cdot Hdxd\tau\Big|\le CM_0\Big(\intoxt|H|^2dxd\tau\Big)^\frac{1}{2},
\end{align*}
where $M_0$ is given by
\begin{align*}
M_0=\|\rho_0-\bar{\rho}_0\|_{L^2}+\|\rho_0u_0-\bar{\rho}_0\bar{u}_0\|_{L^2}+T^{\delta}\Big(\intoxt|z|^2dxd\tau\Big)^\frac{1}{2}
\end{align*}
for some $\delta>0$, and $C>0$ is now fixed. Following the analysis given in Hoff \cite[pp. 1758-1759]{hoff06}, there exists a small time $\tilde\tau>0$ such that 
\begin{equation*}
\Big(\int_0^{\tilde\tau}\!\!\!\intox|z|^2dxd\tau\Big)^\frac{1}{2}\le 2CM_0,
\end{equation*}
and consequently
\begin{equation*}
|\mathcal{R}_1(0,\tilde\tau)|+|\mathcal{R}_6(0,\tilde\tau)|\le M_0\Big(\int_0^{\tilde\tau}\!\!\!\intox|G|^2dxd\tau\Big)^\frac{1}{2}.
\end{equation*}
By applying \eqref{estimate on z in T} with $T$ replaced by $2\tilde\tau$, we get
\begin{equation*}
\Big(\int_0^{2\tilde\tau}\!\!\!\intox|z|^2dxd\tau\Big)^\frac{1}{2}\le 4CM_0.
\end{equation*}
Since $\tilde\tau>0$ is fixed, we can exhaust the interval $[0,T]$ in finitely many steps to obtain that
\begin{equation*}
\Big(\intoxt|z|^2dxd\tau\Big)^\frac{1}{2}\le CM_0,
\end{equation*}
for some new constant $C>0$. Hence the term $T^{\delta}\Big(\intoxt|z|^2dxd\tau\Big)^\frac{1}{2}$ can be eliminated from the definition of $M_0$ by a Gronw\"{a}ll-type argument. Therefore we conclude that
\begin{align}
\Big|\intoxt z\cdot Gdxd\tau\Big|&\le CM_0\Big(\intoxt|G|^2dxd\tau\Big)^\frac{1}{2},\label{L2 bound on u uniqueness}\\
\Big|\intoxt (B-\Bar{B})\cdot Hdxd\tau\Big|&\le CM_0\Big(\intoxt|H|^2dxd\tau\Big)^\frac{1}{2}\label{L2 bound on B uniqueness}.
\end{align}
Since \eqref{L2 bound on u uniqueness} and \eqref{L2 bound on B uniqueness} hold for any $G,H\in H^\infty(\R^3\times[0,T])$, it shows that both $\|z\|_{L^2([0,T]\times\R^3)}$ and $\|B-\bar{B}\|_{L^2([0,T]\times\R^3)}$ are bounded by $M_0$. Finally, using the bound \eqref{bound on time integral on nabla u and bar u} on $\nabla\bar{u}$,
\begin{align*}
\intoxt|\bar{u}-\bar{u}\circ S|^2dxd\tau&\le\int_0^T\|\nabla\bar{u}(\cdot,\tau)\|^2_{L^\infty}\intox|x-S(x,\tau)|^2dxd\tau\\
&\le C\intoxt|z|^2dxd\tau,
\end{align*}
and hence \eqref{estimate on difference of solutions} follows. We finish the proof of Theorem~\ref{Uniqueness theorem}.



\end{document}